\documentclass[12pt,a4paper,openbib]{article}
\usepackage[utf8]{inputenc}
\usepackage[T1]{fontenc}
\usepackage[french]{babel}
\usepackage{amsmath}
\usepackage{amsfonts}
\usepackage{dsfont}
\usepackage{tikz}
\usetikzlibrary{cd}
\usepackage{authblk}
\usepackage{amssymb}
\usepackage{amsthm}
\usepackage{lmodern}
\usepackage{fancyhdr}
\usepackage{lettrine}
\usepackage{bm}
\usepackage{enumerate}
\usepackage{multicol}
\usepackage[hidelinks]{hyperref}
\usepackage[nottoc, notlof, notlot]{tocbibind}
\usepackage[left=2cm,right=2cm,top=2cm,bottom=2cm]{geometry} 
\usepackage{graphicx}
\usepackage{mathtools}
\usepackage{wasysym}
\usepackage{stmaryrd}
\usepackage[all,cmtip]{xy}
\usepackage{enumitem}
\usepackage{relsize} 
\usepackage{mathrsfs}

\makeatletter

\newcommand{\address}[1]{\gdef\@address{#1}}
\newcommand{\email}[1]{\gdef\@email{\url{#1}}}
\newcommand{\@endstuff}{\par\vspace{\baselineskip}\noindent\small
\begin{tabular}{@{}l}\scshape\@address\\\textit{E-mail address:} \@email\end{tabular}}
\AtEndDocument{\@endstuff}
\makeatother

\author{Damien Junger\footnote{This work has been written in a great part during the author PhD thesis at the ENS Lyon. His work are currently funded by the Deutsche Forschungsgemeinschaft (DFG, German Research Foundation) under Germany's Excellence Strategy EXC 2044–390685587, Mathematics Münster: Dynamics–Geometry–Structure.}}
\address{Mathematisches Institut, Universität Münster,\\ Fachbereich Mathematik und Informatik der Universität Münster,  Orléans-Ring 10, 48149 Münster, Germany.}
\email{djunger@uni-muenster.de}
\title{Cohomologie analytique des arrangements d'hyperplans}

\newtheorem{theointro}{Th\'eor\`eme}

\newtheorem{remintro}{Remarque}

\newtheorem{theo}{Th\'eor\`eme}[section]
\newtheorem{lem}[theo]{Lemme}
\newtheorem{coro}[theo]{Corollaire}
\newtheorem{prop}[theo]{Proposition}

\theoremstyle{definition}
\newtheorem{defi}[theo]{D\'efinition}

\theoremstyle{remark}
\newtheorem{rem}[theo]{Remarque}

\newtheorem{ex}[theo]{Exemple}

\DeclareMathOperator{\spec}{Spec}

\DeclareMathOperator{\spg}{Sp}

\DeclareMathOperator{\gln}{GL}

\DeclareMathOperator{\modut}{-Mod}

\DeclareMathOperator{\hhh}{H}
\DeclareMathOperator{\rrr}{R}

\DeclareMathOperator{\pic}{Pic}

\font\tengoth=eufb10
\font\sevengoth=eufb7
\font\fivegoth=eufb5

\newfam\gothfam
\textfont\gothfam=\tengoth
\scriptfont\gothfam=\sevengoth
\scriptscriptfont\gothfam=\fivegoth

\def\A{{\mathbb{A}}}
\def\B{{\mathbb{B}}}

\def\F{{\mathbb{F}}}
\def\G{{\mathbb{G}}}
\def\H{{\mathbb{H}}}

\def\N{{\mathbb{N}}}

\def\P{{\mathbb{P}}}
\def\Q{{\mathbb{Q}}}
\def\R{{\mathbb{R}}}

\def\Z{{\mathbb{Z}}}

\def\AC{{\mathcal{A}}}
\def\BC{{\mathcal{B}}}
\def\CC{{\mathcal{C}}}

\def\HC{{\mathcal{H}}}

\def\KC{{\mathcal{K}}}

\def\OC{{\mathcal{O}}}

\def\UC{{\mathcal{U}}}
\def\VC{{\mathcal{V}}}

\def\XG{{\mathfrak{X}}}

\def\mG{{\mathfrak{m}}}

\def\Ff{{\mathscr{F}}}

\def\Lf{{\mathscr{L}}}

\def\Of{{\mathscr{O}}}

\def\bar#1{\overline{#1}}
\def\het#1{{\rm H}^{#1}_{\rm ét}}

\def\hzar#1{{\rm H}^{#1}_{\rm zar}}

\def\han#1{{\rm H}^{#1}_{\rm an}}

\def\hdr#1{{\rm H}^{#1}_{\rm dR}}

\def\hcech#1{\check{\rm H}^{#1}}
\def\ccech#1{\check{\CC}^{#1}}

\def\et{\text{ et }}
\def\si{\text{Si }}
\def\sinon{\text{Sinon }}
\def\and{\text{ and }}

\def\fln#1#2{\xrightarrow[#2]{#1}}

\def\flsur{\twoheadrightarrow}
\def\fla{\mapsto}

\def\limp{\varprojlim}

\def\som#1#2#3{\sum\limits_{{\substack{#2}}}^{#3}{#1}}

\def\inter#1#2#3{\bigcap\limits_{{\substack{#2}}}^{#3}{#1}}
\def\uni#1#2#3{\bigcup\limits_{{\substack{#2}}}^{#3}{#1}}

\def\drt#1#2#3{\bigoplus\limits_{{\substack{#2}}}^{#3}{#1}}

\begin{document}

\maketitle

\begin{abstract}
In this article, we study the cohomology of some analytic sheaves on the complementary in the projective space of a suitable infinite collection of hyperplane like the Drinfel'd symetric space. In particular, the sheaf of invertible functions on these rigid spaces has no cohomology in degree greater or equal to $1$. This proves the vanishing of the Picard goup and the methods used give a convenient description of the global invertible functions.
\end{abstract}

\tableofcontents
 
\section*{Introduction}

   Soit $p$ un nombre premier. Cet article est li\'e \`a une s\'erie de travaux r\'ecents
   \cite{GPW1,GPW2,GPW3}, portant sur la g\'eom\'etrie et la cohomologie 
   $p$-adique des espaces sym\'etriques de Drinfeld et de leurs rev\^etements. 
   Ces espaces sont des cas particuliers d'arrangements (infinis) d'hyperplans et l'objet 
   de cet article est de comprendre ce qui se passe pour des arrangements plus g\'en\'eraux. 
   L'\'etude de leur cohomologie \'etale $p$-adique semble d\'elicate (en effet, les travaux cit\'es utilisent des propri\'et\'es sp\'ecifiques de l'espace de Drinfeld), nous nous int\'eresserons plut\^ot \`a leur cohomologie analytique \`a coefficients dans le faisceau $\G_m$ des fonctions inversibles. Notre r\'esultat principal affirme que beaucoup d'arrangements  (m\^emes infinis) d'hyperplans sont acycliques pour $\G_m$. Par exemple, cela entra\^ine que le groupe de Picard des espaces de Drinfeld est trivial, ce qui ne semble pas  être connu. 
   Il serait int\'eressant d'avoir des r\'esultats analogues pour la cohomologie \'etale,  mais cela nous semble inaccessible pour le moment. En effet, le calcul de $\het{2}(X, \G_m)$ pour l'espace de Drinfeld $X$ de dimension plus grande que $1$ semble d\'ej\`a d\'elicat (la partie de torsion est cependant bien comprise gr\^ace aux r\'esultats de Schneider-Stuhler \cite{scst} et Colmez-Dospinescu-Niziol \cite{GPW3}). 
  
    Avant de pr\'eciser nos r\'esultats principaux, mentionnons-en certaines motivations et applications \`a l'\'etude du premier rev\^etement des espaces de Drinfeld. 
Soit $K$ une extension finie de $\Q_p$, $\OC_K$ son anneau d'entiers, $\F=\F_q$ son corps r\'esiduel et $\varpi$ une uniformisante. Soit $C$ le compl\'et\'e d'une cl\^oture alg\'ebrique de $K$.
On note $\H_K^d$ l'espace sym\'etrique de Drinfeld de dimension $d\geq 1$, i.e. l'espace rigide analytique sur $K$ défini par\footnote{Il n'est pas immédiat que ce complémentaire d'un nombre infini de parties fermées est bien un espace rigide analytique mais cela découle de la remarque \ref{remarr}.} 
$$\H_K^d=\mathbb{P}^d_K\setminus \bigcup_{H\in \mathcal{H}} H,$$
avec $\mathcal{H}$ l'ensemble des hyperplans $K$-rationnels et $\mathbb{P}^d_K$ l'espace projectif rigide analytique de dimension $d$ sur $K$. L'espace $\H_K^d$ poss\`ede un mod\`ele formel 
semi-stable $\H_{\OC_K}^d$, construit par Deligne. Soit $D$ l'alg\`ebre \`a division sur $K$ d'invariant $\frac{1}{d+1}$ et $\Pi_D$ une uniformisante, un théorème fondamental de Drinfeld \cite{dr2} fournit une interpr\'etation modulaire de l'espace $\H^d_{\OC_K}$, et cette description entraîne l'existence d'un $\OC_D$-module formel universel $\XG$ sur $\H_{\OC_K}^d$. Les points de $\Pi_D$-torsion $\XG[\Pi_D]$ forment un sch\'ema formel en $\F_p$-espaces vectoriels de Raynaud. Ces derniers admettent une classification \cite{Rayn} et sont caract\'eris\'es par la donn\'ee des parties isotypiques $(\Lf_i)_{i \in \Z/(d+1)\Z}$ de $\Of(\XG[\Pi_D])$ pour certains caract\`eres de $\F_{q^{d+1}}$, dits fondamentaux. Comprendre les  fibr\'es en droites  $(\Lf_i)_i$ universels sur $\H_{\OC_K}^d$ est essentiel pour comprendre la g\'eom\'etrie   du premier rev\^etement $\Sigma^1$ de $\H_K^d$.
En fibre sp\'eciale, les  faisceaux $(\Lf_i)_i$ sont relativement bien compris et sont \'etudi\'es dans \cite{teit1}, \cite{teit2}, \cite{teit4} et \cite{GK6}. L'annulation du groupe de Picard de $\H_K^d$, qui d\'ecoule de nos r\'esultats, fournit donc une description en fibre g\'en\'erique de ces faisceaux localement libres de rang $1$ sur $\H_{\OC_K}^d$. Dans un travail ult\'erieur, nous obtiendrons une classification des $\mu_N$-torseurs sur $\H_K^d$ avec $N=q^{d+1}-1$ et nous donnerons une \'equation explicite du rev\^etement mod\'er\'e de l'espace sym\'etrique de Drinfeld. 




   Passons maintenant \`a notre r\'esultat principal. Gardons les notations ci-dessus. 
   Soit $\mathcal{A}$ une partie ferm\'ee (par exemple une partie finie) de l'espace profini
   $\mathcal{H}$ et posons 
   $${\rm Int}(\mathcal{A})=\P^d_{K}\setminus \bigcup_{H\in \mathcal{A}} H.$$
   Alors ${\rm Int}(\mathcal{A})$ poss\`ede encore une structure naturelle d'espace rigide analytique sur $K$. 
   Si $L$ est une extension compl\`ete de $K$ et 
   si $X$ est un $K$-espace analytique, on note $X_L=X\hat{\otimes}_K L$.
   
   \begin{theointro}
     Avec les notations ci-dessus, pour toute partie ferm\'ee $\mathcal{A}$ de 
     $\mathcal{H}$ et toute extension compl\`ete $L$ de $K$ on a $\han{i}({\rm Int}(\mathcal{A})_L, \G_m)=0$ pour $i\geq 1$. 
   \end{theointro}

   \begin{remintro}
   
   \begin{enumerate}
   
\item L'\'egalit\'e \cite[proposition 4.1.10]{ber2}
\[ \han{1}(X, \G_m) = \het{1}(X, \G_m)=\pic(X) \]
est valable 
pour tout espace analytique $X$. Ainsi le groupe de Picard et les fonctions inversibles sur $X$ peuvent \^etre d\'etermin\'ees en calculant sa cohomologie analytique. Ce n'est malheureusement pas le cas
des groupes de cohomologie en degr\'es strictement plus grands que $1$. 
 
  \item Nous prouvons aussi une version du th\'eor\`eme dans laquelle le faisceau $\G_m$ est remplac\'e par le sous-faisceau $\Of^{**}=1+\Of^{++}$ des fonctions $1+f$ telles que $|f|<1$ (la norme \'etant celle spectrale). Si $\Of^+$ d\'esigne le faisceau des fonctions $f$ telles que $|f|\leq 1$, il est probable que $\han{i}({\rm Int}(\mathcal{A})_L, \Of^+)=0$ pour $i\geq 1$, mais nous n'arrivons pas \`a le d\'emontrer. Le r\'esultat analogue avec ${\rm Int}(\mathcal{A})_L$ remplac\'e par une boule ferm\'ee est un th\'eor\`eme de Bartenwerfer \cite{bart} (il est crucial d'utiliser la topologie analytique pour ce genre de r\'esultat car il est totalement faux pour la topologie \'etale). Nos m\'ethodes permettent de d\'emontrer que si le r\'esultat de Bartenwerfer est aussi valable pour les poly-couronnes, alors $\han{i}({\rm Int}(\mathcal{A})_L, \Of^+)=0$ pour $i\geq 1$. 
     
     \end{enumerate}
 \end{remintro}

   Notons $\Z\left\llbracket \mathcal{A}\right\rrbracket $ le dual du $\Z$-module ${\rm LC}(\mathcal{A}, \Z)$ des fonctions 
   localement constantes sur $\mathcal{A}$, \`a valeurs dans $\Z$. On voit les éléments de $\Z\left\llbracket \mathcal{A}\right\rrbracket $ comme des mesures sur $\mathcal{A}$, à valeurs dans $\Z$. On note 
   $\Z\left\llbracket \mathcal{A}\right\rrbracket^0$ le sous-groupe des mesures de masse totale $0$ (i.e. l'orthogonal de la fonction constante $1$).
   
   \begin{theointro}
    Pour toute partie ferm\'ee $\mathcal{A}$ de 
     $\mathcal{H}$ et toute extension compl\`ete $L$ de $K$ il existe un isomorphisme naturel 
   $$\Of^*({\rm Int}(\mathcal{A})_L)/L^*\simeq \Z\left\llbracket \mathcal{A}\right\rrbracket^0.$$
    \end{theointro}

   \begin{remintro}
   \begin{enumerate}
   
    \item Ce th\'eor\`eme a été récemment obtenu par Gekeler \cite{gek} pour l'espace symétrique de Drinfeld. Notre méthode est compl\`etement diff\'erente. 
    
    \item Si l'on combine le th\'eor\`eme ci-dessus avec la suite exacte de Kummer et l'annulation du groupe de Picard, on obtient une description du groupe $\het{1}({\rm Int}(\mathcal{A})_L, \Z/n\Z)$ pour tout entier $n$. Cela semble sugg\'erer qu'il existe des descriptions explicites de la cohomologie \'etale en degr\'e cohomologique plus grand. Voir \cite{GPW3} pour le cas de l'espace de Drinfeld.

   \end{enumerate}
   \end{remintro}
    
    Nous finissons cette introduction en expliquant les grandes \'etapes de la preuve de nos r\'esultats principaux. L'ingrédient technique principal est un résultat d'annulation de van der Put \cite{vdp}, qui affirme que pour tout $r\in p^{\mathbb{Q}}$ le faisceau $\Of^{(r)}$ des fonctions de norme spectrale strictement plus petite que $r$ est acyclique sur les boules fermées et les polycouronnes de dimension arbitraire. Pour se ramener à ce type d'espaces nous utilisons les constructions géométriques de Schneider et Stuhler \cite{scst}. Plus précisément, 
    l'espace ${\rm Int}(\mathcal{A})$ possède un recouvrement de type Stein par des affinoides 
    ${\rm Int}(\mathcal{A}_n)$ obtenus en enlevant de $\P^d_K$ les tubes ouverts d'épaisseur 
    $|\varpi|^n$ autour des hyperplans dans $\mathcal{A}$. Cela nous amène à étudier la géométrie d'un arrangement tubulaire 
    $$X_I=\P^d_K\setminus \bigcup_{i\in I} H_i(|\varpi|^n),$$
    où $H_i(|\varpi|^n)$ est le voisinage tubulaire ouvert d'épaisseur $|\varpi|^n$ de l'hyperplan $H_i$. Nous allons supposer que ces voisinages tubulaires sont deux à deux distincts. Suivant Schneider et Stuhler, pour comprendre la géométrie de $X$, il s'agit de comprendre la géométrie des espaces de la forme 
    $$Y_J=\P^d_K\setminus \bigcap_{j\in J} H_j(|\varpi|^n)$$
    avec $J\subset I$. Le point essentiel est que les espaces $Y_J$ sont des fibrations localement triviales 
    en boules fermées au-dessus d'espaces projectifs (dont la dimension dépend de la combinatoire des hyperplans). Cela permet d'utiliser les résultats d'annulation de van der Put et nous ramène à l'étude de certains complexes de Cech relativement explicites. Pour transférer l'étude des faisceaux sur les $Y_J$ à $X_I$ nous montrons un lemme combinatoire (essentiellement basé sur la suite de Mayer-Vietoris), qui remplace la suite spectrale utilisée par Schneider et Stuhler (et dont l'étude devient assez compliquée dans notre situation). Cela permet de démontrer que les faisceaux $\Of^{(r)}$ sont acycliques sur $X_I$. Un argument basé sur le logarithme tronqué permet d'en déduire l'acyclicité du faisceau 
    $\Of^{**}=1+\Of^{++}$ des fonctions $1+f$ vérifiant $|f|<1$ sur les $X_I$. Enfin, l'étude du quotient 
    $\G_m/\Of^{**}$ fait apparaître des complexes de Cech identiques à ceux apparaissant en géométrie algébrique, ce qui permet de passer de $\Of^{**}$ à $\G_m$.

     Le paragraphe précédent explique la preuve de l'acyclicité de $\G_m$ sur les espaces $X_I$.  Le passage de ces espaces à ${\rm Int}(\mathcal{A})$ n'est pas trivial et représente en fait le coeur technique de l'article. Pour expliquer la difficulté, notons que l'on dispose d'un recouvrement Stein ${\rm Int}(\mathcal{A})=\cup_{n\geq 1} X_{I_n}$, où les $X_{I_n}$ sont des espaces du même type que ceux introduits ci-dessus, les $I_n$ étant des ensembles finis, de plus en plus grands. On en déduit une suite exacte 
     $$0\to \rrr^1\varprojlim_n \han{s-1}(X_{I_n}, \G_m)\to \han{s}({\rm Int}(\mathcal{A}), \G_m)\to 
     \varprojlim_n \han{s}(X_{I_n}, \G_m)\to 0.$$
     Pour $s>1$ cela permet de démontrer l'annulation de $\han{s}({\rm Int}(\mathcal{A}), \G_m)$, mais pour 
     $s=1$ il s'agit de démontrer que 
     $$R^1\varprojlim_n \Of^*(X_{I_n})=0.$$
    Pour cela, on se ramène à démontrer le même résultat avec les faisceaux 
    $\Of^{**}$ et $\Of^{(r)}$ à la place de $\G_m$. Le point crucial à démontrer est alors une version en dimension quelconque 
    du lemme 1.12 de \cite{GPW1}, qui permet de comprendre la flèche de restriction 
    $\Of^{**}(X_{I_{n+1}})\to \Of^{**}(X_{I_n})$. Plus précisément, par application du logarithme, nous nous ramenons à montrer pour $r$ assez petit, qu'il existe une constante 
    $c$ telle que, pour tout $n$,  $ \Of^{(r)}(X_{I_{n+c}})\subset \OC_L^{(r)}+ \varpi \Of^{(r)}(X_{I_n})$ avec $\OC_L^{(r)}$ le sous-ensemble des éléments de $L$ de norme strictement inférieure ou égale à $r$. C'est le point le plus délicat de l'article et la preuve en est assez indirecte car nous n'avons pas de description explicite des groupes $\Of^{(r)}(X_{I_n})$.

\subsection*{Notations et conventions}


Dans tout l'article, on fixe un nombre premier $p$ et une extension finie $K$ de $\Q_p$. On note $\mathcal{O}_K$ son anneau des entiers, $\varpi$ une uniformisante et $\F=\F_q$ son corps r\'esiduel. On note $C=\hat{\bar{K}}$ la complétion d'une clôture algébrique de $K$ et $\breve{K}$ la complétion de l'extension maximale non ramifiée de $K$. Soit $L\subset C$ une extension complète de $K$ susceptible de varier (par exemple $K,\breve{K},C$), d'anneau des entiers $\mathcal{O}_L$, d'idéal maximal  $\mG_L$ et de corps r\'esiduel $\kappa$. 

Soit $S$ un $L$-espace rigide analytique\footnote{Dans tout le reste de l'article, les espaces rigides analytiques et les affinoides seront supposés "à la Tate". Ce cadre sera largement suffisant pour nos applications.}. On note $\A^n_{rig, S}$ (respectivement $\P_{rig, S}^n$) l'espace affine (resp. projectif) rigide analytique de dimension relative $n$ sur $S$. Si 
 $s=(s_i)_{1\le i\le n}$ est une famille de nombres rationnels, le polydisque rigide fermé sur $S$ de polyrayon $(|\varpi|^{s_i})_i$ 
  sera noté $\B_S^n  (|\varpi|^s)$ ou $\B_S^n  (s)$ par abus. 
  L'espace $\B^n_S$ sera la boule unité et les boules ouvertes seront notées  $\mathring{\B}^n_S$ et $\mathring{\B}_S^n (s)$. Si $S$ est maintenant un schéma,  $\A^n_{zar, S}$ sera l'espace affine sur $S$ et $\P^n_{zar, S}$ l'espace projectif.
  
Si $X$ est un $L$-espace analytique réduit, on note $\Of^+_X$ le faisceau des fonctions à puissances bornées, $\Of^{++}_X$ le faisceau des fonctions topologiquement nilpotentes,  $\Of^{(r)}_X$ le faisceau des fonctions bornées strictement en norme spectrale par $r$, $\Of^{*}_{X}$ (ou bien $\G_{m,X}$) le faisceau des fonctions inversibles  et $\Of^{**}_{X}$ le faisceau $1+ \Of^{++}_X$. 
Si $X=\spg(L)$, on écrit $\OC^{(r)}_L=\Of^{(r)}_X (X)$. 
Pour tout ouvert affinoïde réduit $U\subset X$,  on munit $\Of^*_X (U)$ de la topologie induite par le plongement $\Of^*_X(U)\to\Of_X (U)^2 : f\fla  (f,f^{-1})$ (muni de la norme spectrale). 
On notera  $K(x)$ le corps valué associé au point fermé $x\in X$.



Si $X$ est un espace analytique sur $L$ (resp. un schéma), la cohomologie d'un faisceau $\Ff$ sur le site analytique (resp. de Zariski) sera notée $\han{*} (X,\Ff)$ (resp. $\hzar{*} (X,\Ff)$).  Si $\UC$ est un recouvrement de $X$ (pour une des topologies précédemment nommées), la cohomologie de Cech de $X$ pour le faisceau $\Ff$ par rapport au recouvrement $\UC$ sera notée $\hcech{*} (X,\Ff, \UC)$ et le complexe de cochaînes sera noté $\check{\CC}^* (X,\Ff,\UC)$. Pour toutes ces théories cohomologiques, quand $U\subset X$ est un ouvert de $X$, la cohomologie à  support dans le complémentaire de $U$ sera notée $\hhh^*(X,U)$.  Si $ \Lambda$ est un groupe cyclique d'ordre $N$ premier à  $p$ et $\bar{X}=X\hat{\otimes}C$, le morphisme de Kummer  sera noté $\kappa : \Of^* (X)\to  \het{1}(X,\Lambda_X)$ et $\bar{\kappa} :  \Of^* (X)\to  \het{1}(\bar{X},\Lambda_{\bar{X}})$ sera la restriction de $\Of^* (\bar{X})\to  \het{1}(\bar{X},\Lambda_{\bar{X}})$.


Enfin, nous noterons  $\left\llbracket a,b\right\rrbracket:= [a,b]\cap \Z$ quand $a,b\in \R$.

\subsection*{Remerciements}

Le présent travail a été, avec \cite{J2,J3,J4}, en grande partie réalisé durant ma thèse à l'ENS de lyon,  et a pu bénéficier de la relecture attentive de mes maîtres de thèse Vincent Pilloni et Gabriel Dospinescu qui ont beaucoup apporté à la clarté et à la rigueur de l'exposition. Je leur en suis très reconnaissant.  Je tenais aussi à remercier Najmuddin Fakhruddin pour m'avoir suggéré la preuve de \ref{propptzar},  Sophie Morel pour les discussions sur les travaux de  Lütkebohmert et le referee pour leurs précieux conseils. Enfin, le support "logistique" réalisé par Sally Gilles et Juan Esteban Rodriguez  Camargo ont rendu possible cette article. 

\section{L'espace des hyperplans $K$-rationnels} 

   On note $\HC$ l'ensemble des hyperplans $K$-rationnels dans $\P^d$.  L'ensemble $\mathcal{H}$ est profini car il s'identifie à $\P^d(K)$.






   
 %

Définissons maintenant quelques données relatives à l'ensemble $\HC$.
   Si $a=(a_0,\dots, a_d)\in C^{d+1}\backslash \{0\}$, $l_a$ désignera l'application \[b=(b_0,\dots, b_d)\in C^{d+1} \mapsto \left\langle a,b\right\rangle := \som{a_i b_i}{0\leq i\leq d}{}.\] Ainsi $\HC$ s'identifie à  $\{\ker (l_a),\; a\in K^{d+1}\backslash \{0\} \}$ et à  $\P^d (K)$.

\begin{rem}

L'application précédente permet d'identifier un hyperplan dans $\HC$ à sa droite orthogonale par dualité. Nous confondrons alors toujours un élément de $\HC$ à sa droite associée. Dans la section suivante, nous attacherons des espaces rigides à certaines parties de $\HC$ et nous pourrons décrire explicitement leur géométrie et leur combinatoire uniquement grâce aux relations linéaires sur $\OC_K$ entre les générateurs unimodulaires de ces droites.

\end{rem}
   
   Le vecteur $a=(a_i)_{i}\in C^{d+1}$ est dit unimodulaire si $|a|_{\infty}(:=\max (|a_i|))=1$. L'application 
   $a\mapsto H_a:=\ker (l_a)$ induit une bijection entre le quotient de l'ensemble des vecteurs unimodulaires 
   $a\in K^{d+1}$ par l'action évidente de $\OC_K^*$ et l'ensemble $\mathcal{H}$. 
   
   Pour $a\in K^{d+1}$ unimodulaire et $n\geq 1$, 
   on considère l'application $l_a^{(n)} $ \[b\in (\OC_{C}/\varpi^n)^{d+1}\mapsto \left\langle a,b\right\rangle\in \OC_{C}/\varpi^n\] et on note $$\HC_n =\{\ker (l_a^{(n)}),\; a\in K^{d+1}\backslash \{0\} \; {\rm unimodulaire} \}\simeq\P^d(\OC_K/\varpi^n).$$ Alors $\HC = \varprojlim_n \HC_n$ et chaque 
   $\HC_n$ est fini.


Soit $a\in K^{d+1}$ unimodulaire et $z\in \P^d (C)$. La quantité $|l_a (b)| $ ne dépend pas du choix du 
représentant unimodulaire $b$ de $z$, et ne dépend que de la classe de $a$ dans $\P^d (K)$. Cela permet de définir les  tubes fermés et ouverts de rayon $\varepsilon>0$ autour de l'hyperplan $H=\ker(l_{a_H})\in\HC$ par \[\bar{H}(\varepsilon) =\{z\in\P^d(C), |l_{a_H}(z)|\leq\varepsilon \} \et \mathring{H}(\varepsilon) =\{z\in\P^d(C), |l_{a_H}(z)|<\varepsilon \}. \]  

Les extensions des scalaires par $L$ seront notées $\bar{H}(\varepsilon)_L$ et $\mathring{H}(\varepsilon)_L$ et les complémentaires dans $\P^d_{rig,L}$ seront $\bar{H}(\varepsilon)^c_L$ et $\mathring{H}(\varepsilon)^c_L$. Il est à  noter que $\bar{H}(|\varpi|^n)$ (resp. $\mathring{H}(|\varpi|^n)$) ne dépend que de la classe de $H$ dans $\HC_{n }$ (resp. $\HC_{n+1}$). \footnote{Cela découle du fait que pour deux vecteurs $a_1$, $a_2$, on a l'identité $|l_{a_1}(z)-l_{a_2}(z)|\le |a_1-a_2|_{\infty}|z|_{\infty}$. En particulier, les inégalités $|l_{a_1}(z)|\le \varepsilon$ et $|l_{a_2}(z)|\le \varepsilon$ (resp. $|l_{a_1}(z)|< \varepsilon$ et $|l_{a_2}(z)|< \varepsilon$) sont vérifiées pour les mêmes vecteurs unimodulaires $z$ si $|a_1-a_2|_{\infty}\le \varepsilon$ (resp. $|a_1-a_2|_{\infty}< \varepsilon$).}
 
\begin{rem}
\label{remgH}
Dans la définition des tubes ouverts ou fermés, nous avons procédé à une renormalisation lorsque nous évaluons $|l_a(z)|$ sur un représentant unimodulaire de $z$.   Cette dernière dépend du choix de coordonnées initiales.   Toutefois,  les changements de variables dans $\gln_{d+1}(\OC_K)$ permutent les tubes de même rayon.  Plus précisément, si $g\in \gln_{d+1}(\OC_K)$,  $g\cdot \bar{H}(\epsilon)= \bar{(gH)}(\epsilon)$.    
\end{rem}

\section{G\'eom\'etrie des arrangements\label{sssectiongeomarr}} 

\subsection{D\'efinitions et exemples}


Pour toute collection 
$\AC$ de parties de $\P^d_{rig, K}$ on note 
$$ {\rm Int} (\AC)=\P^d_{rig, K} \backslash \bigcup_{H \in \AC} H$$
et
$${\rm Uni} (\AC)=\P^d_{rig, K} \backslash \bigcap_{H \in \AC} H.$$

Dans le cas général, ces constructions n'admettent pas forcément de structure naturelle d'espaces rigides analytiques. Toutefois, c'est le cas lorsque $\AC$ est finie et constituée de parties fermées. Nous nous intéresserons dans la suite de l'article aux exemples suivants (où est $\AC$ possiblement infini) pour lesquels une telle structure d'espace rigide existe. 

\begin{defi}\label{defarr} Une collection $\AC$ de parties de $\P^d_{rig, K} $ est appel\'ee 

$\bullet$ {\it arrangement alg\'ebrique (resp. alg\'ebrique g\'en\'eralis\'e)} (d'hyperplans $K$-rationnels) si $\AC$ est un sous-ensemble fini (resp. ferm\'e) de $\HC$. 

$\bullet$ {\it arrangement tubulaire ouvert (resp. ferm\'e) d'ordre $n$} si $\AC$ est une famille finie de voisinages tubulaires  {\it ferm\'es (resp. ouverts)} $\bar{H}( |\varpi|^n)$ (resp. $\mathring{H}( |\varpi|^n)$) avec $H\in \HC$.

\end{defi}

\begin{rem}\label{remarr}

\begin{itemize}

\item Pour simplifier l'exposition, nous avons choisi d'étudier les arrangements pour des hyperplans $K$-rationnels pour $K$ une extension finie de $\Q_p$. Certaines de ces constructions peuvent s'étendre à des corps $K$ beaucoup plus généraux. Par exemple, les arrangements tubulaires (ouverts ou fermés) peuvent être définies sur n'importe quel corps complet. Toutefois, les arguments que nous allons présenter se servent de manière cruciale de l'existence de fibrations dont la construction, que nous rappelons dans \ref{ssectionfibr}, nécessite de supposer  $K$ de valuation discrète. Pour le cas des arrangements algébriques généralisés, il est nécessaire d'imposer en plus la finitude du corps résiduel de $K$ pour que soit bien un espace rigide (cf la note \ref{notefin} et le troisième point de cette remarque). En caractéristique positive, ces constructions peuvent être transportées mutatis mutandis modulo les hypothèses sur  $K$ précédente.  En revanche, certains résultats d'annulation cohomologique que nous allons énoncer ne peuvent être prouvés dans ce cadre grâce aux méthodes de l'article. Nous renvoyons à la remarque \ref{remgenprinc} pour une discussion plus précise.

\item Se donner un arrangement tubulaire ouvert (resp. ferm\'e) d'ordre $n$ revient \`a se donner une partie (finie)\footnote{\label{notefin} L'hypothèse de finitude est ici redondante sous les conditions que nous avons imposé sur le corps $K$ car $\HC_n$ est fini pour tout $n$. Cette propriété n'est plus vrai si on raisonne sur un corps plus général dont le corps résiduel peut être infini ou qui peut ne pas être de valuation discrète.} de $\HC_{n}$ (resp. $\HC_{n+1}$). 

\item Si $m>n$, tout arrangement tubulaire ouvert (ou ferm\'e) d'ordre $m$ induit un arrangement 
tubulaire ouvert (ou ferm\'e) d'ordre $n$, appel\'e sa projection. Plus pr\'ecis\'ement, la projection d'un arrangement d\'efini par une collection de voisinages tubulaires $(\bar{H}( |\varpi|^m))_{H\in I}$ 
est l'arrangement d\'efini par la collection de voisinages tubulaires $(\bar{H}( |\varpi|^n))_{H\in I}$. 
Cela revient \`a consid\'erer la projection d'une partie de  $\HC_{m}$ (resp. $\HC_{m+1}$) sur 
 $\HC_{n}$ (resp. $\HC_{n+1}$). Cette construction s'\'etend bien s\^ur au cas d'un arrangement alg\'ebrique (resp. alg\'ebrique 
g\'en\'eralis\'e).\footnote{Dans le cas des arrangements algébriques généralisés, il est nécessaire de demander à ce que $K$ soit de valuation discrète et de corps résiduel fini pour que les projections d'ordre $n\in\N$ forment bien une famille finie de voisinages tubulaires.} 

\item Une famille d'arrangements tubulaires $(\AC_n)_n$ 
telle que l'ordre de $\AC_n$ soit $n$ 
est dite compatible si, pour tout $m>n$, $\AC_n$ est la projection de $\AC_m$.
Si $\AC\subset \HC$ est un arrangement alg\'ebrique g\'en\'eralis\'e, on construit par projection deux familles compatibles d'arrangements tubulaires ouverts (resp. ferm\'es) $(\AC_n)_n$ par projection. Cette construction a pour intérêt de fournir un recouvrement croissant ${\rm Int} (\AC) =\bigcup_n {\rm Int} (\AC_n)$ pour les arrangements algébriques généralisés $\AC$.

\item Les constructions $ {\rm Int} (\AC)$ (et ${\rm Uni} (\AC)$ lorque $\AC$ est fini) poss\`edent des structures naturelles 
d'espaces rigides analytiques sur $K$. Le seul cas non trivial est celui d'un arrangement alg\'ebrique g\'en\'eralis\'e, qui d\'ecoule du point pr\'ec\'edent. 

\item Plus pr\'ecis\'ement, l'espace ${\rm Int} (\AC)$ est un affinoïde (resp. quasi-Stein) si $\AC$ est un arrangement tubulaire fermé (resp. ouvert).

\end{itemize}
\end{rem}

\begin{ex}
L'espace symétrique de Drinfeld $\H^d_K$ est l'arrangement d'hyperplan généralisé  ${\rm Int}(\HC)$.
\end{ex}
Nous allons définir le rang d'un arrangement $\AC$, qui permettra de décrire la géométrie de ${\rm Uni} (\AC)$

\begin{defi}\label{defirgarr}

Nous donnons la notion de rang pour des parties finies de $\HC$ et de $\HC_n$. D'après la deuxième observation de \ref{remarr}, cela induit une notion de rang pour les arrangements algébriques et tubulaires ouverts ou fermés.

\begin{itemize}

\item Si $\AC\subset\HC$, on se donne pour tout $H\in\AC$ un vecteur $a_H$ unimodulaire tel que $H= \ker(l_{a_H})$. On pose $\operatorname{rg}(\AC)= \operatorname{rg}_{\OC_K}( \sum_{H \in \AC} \OC_K a_H)$. 
\item Si $\AC\subset\HC_n$, on se donne pour tout $H$ dans $\AC$, $a_H$ un vecteur unimodulaire dans $\OC_K^{d+1}/ \varpi^{n} \OC_K^{d+1}$ tel que $H= \ker(l_{a_H})$ et $\tilde{a}_H$ un relev\'e dans $\OC_K^{d+1}$. On \'ecrit\footnote{$\alpha_i$ peut être infini et dans ce cas on adopte la convention $\varpi^{\infty}=0$. } \[\sum_{H \in \AC} \OC_K \tilde{a}_H = \bigoplus_{i=0}^d \varpi^{\alpha_i} \OC_K e_i\] pour $(e_i)$ une base de $\OC_K^{d+1}$ bien choisie. On pose alors $\operatorname{rg}(\AC)= \operatorname{card}\{ i : \alpha_i < n\}$. Cette quantit\'e ne d\'epend pas des choix des $a_H$ et de leur relev\'e. Intuitivement, le rang correspond à $\operatorname{rg}(\AC)= \operatorname{rg}_{\OC_K/\varpi^{n}\OC_K}( \sum_{H \in \AC} (\OC_K/\varpi^{n}\OC_K) a_H)$.

\end{itemize}
\end{defi}

\subsection{La suite spectrale associ\'ee \`a un arrangement\label{ssectionfibr} }

Dorénavant, pour tout arrangement d'hyperplans $\AC$, nous verrons ${\rm Int} (\AC)$ et ${\rm Uni}(\AC)$ comme des $L$-espaces analytiques par extension des scalaires. Si $\hhh$ d\'esigne la cohomologie de de Rham ou la cohomologie d'un faisceau $\Ff$ sur le site \'etale ou analytique, on a par un argument g\'en\'eral \cite[§ 2, Proposition 6+Lemma 7 ainsi que les discussions qui précédent]{scst} de suites spectrales :
\begin{equation}\label{eqsuitspectgen}
E_1^{-r,s} = \bigoplus_{(H_i)_{0 \le i \le r} \in \AC^{r+1}} \hhh^s(\P_{rig, L}^d, {\rm Uni} (\{H_i\})) \Rightarrow \hhh^{s-r}(\P_{rig, L}^d, {\rm Int} (\AC)) 
\end{equation}   
o\`u $\AC$ est un arrangement alg\'ebrique, tubulaire d'ordre $n$ ouvert ou ferm\'e et 
$\hhh(X,Y)$ repr\'esente la cohomologie de $X$ \`a support dans $X\setminus Y$. 

Soit $\AC$ un arrangement (algébrique, tubulaire ouvert ou fermé) et $\BC\subset \AC$ non vide de cardinal $r+1$, nous allons donc chercher \`a d\'ecrire la g\'eom\'etrie   de ${\rm Uni} (\BC)$ suivant si $\AC$ est algébrique, tubulaire ouvert ou fermé. Si $r=0$, ${\rm Uni} (\BC)$ devient un espace affine dans le cas algébrique, une boule ouverte dans le cas tubulaire ouvert et une boule fermée dans le cas tubulaire fermé.

Supposons maintenant $r\neq 0$ et posons $t+1={\rm rg}(\BC)$. Par hypothèse, on a $t\neq 0$. 
Nous allons construire en suivant  \cite[§1, Proposition 6]{scst} une fibration $f :{\rm Uni} (\BC)\to \P^t_{rig, L}$. Les fibres seront des espaces affines dans le cas algébrique, des boules ouvertes dans le cas tubulaire ouvert et des boules fermées dans le cas tubulaire fermé. Pour chaque $H_i\in \BC$, choisissons un vecteur unimodulaire de $K^{d+1}$ de la même manière que dans \ref{defirgarr} et écrivons $M:=\som{\OC_K a_i}{0\leq i\leq  r}{}\subset M_0:=\som{\OC_K e_i}{0\leq i\leq d}{}$ où $(e_i)$ est la base canonique de $K^{d+1}$. Réalisons un changement de base similaire à  \ref{defirgarr} (licite d'après \ref{remgH}) pour obtenir des entiers positifs croissants $(\alpha_i)_{0\leq i\leq d}$ 
tels que $\alpha_0=0$ et obtenir une décomposition $M=\som{\varpi^{\alpha_i}\OC_K e_i}{0\leq i\leq d}{}\subset M_0=\som{\OC_K e_i}{0\leq i\leq d}{}$. On a alors les descriptions suivantes de ${\rm Uni} (\BC)$, avec la convention que pour la suite on choisit un repr\'esentant unimodulaire de chaque 
point $[b_0,\cdots, b_d]$, i.e. tel que $\max_{0\leq i\leq d} |b_i|=1$: 

$\bullet$ dans le cas algébrique \[{\rm Uni} (\BC)=Z^d_t:=\{z=[b_0,\cdots,b_d]\in \P_{rig, L}^d, \exists i\leq  t,  b_i\neq 0\},\] 

$\bullet$ dans le cas tubulaire  ferm\'e, 
 posons $\beta=(\beta_i)_{0\leq i\leq t}=(n-\alpha_i)_{0\leq i\leq t}$ et notons 
 \[{\rm Uni} (\BC)=X^d_t(\beta):=\{z=[b_0,\cdots,b_d]\in \P_{rig, L}^d, \exists i\leq  t, | b_i|\geq |\varpi|^{\beta_i} \},\]
 
$\bullet$ dans le cas tubulaire ouvert, 
 posons $\gamma=(\gamma_i)_{0\leq i\leq t}=(n+1-\alpha_i)_{0\leq i\leq t}$ et notons 
 \[{\rm Uni} (\BC)=Y^d_t(\gamma):=\{z=[b_0,\cdots,b_d]\in \P_{rig, L}^d, \exists i\leq  t, | b_i|> |\varpi|^{\gamma_i} \}.\]
 
La flèche   $f $ donnée par $[b_0,\cdots, b_d]\mapsto [b_0,\cdots, b_t]$ induit bien des fibrations\footnote{La flèche $f$ est bien définie sur ces espaces.} $X^d_t(\beta)\to \P_{rig, L}^t$, $Y^d_t(\gamma)\to \P_{rig, L}^t$, $Z^d_t\to \P_{rig, L}^t$. 
Soit $\VC (\beta)=\{V(\beta)_i\}$, $\mathring{\VC }(\gamma)=\{\mathring{V}(\gamma)_i\}$, $\VC=\{V_i\}$  les recouvrements admissibles de $\P_{rig, L}^t$ où \[ V(\beta)_i =\{z=[z_0,\cdots,z_t]\in \P_{rig, L}^t, \forall j\leq t, |\frac{z_i}{\varpi^{\beta_i}}|\geq |\frac{z_j}{\varpi^{\beta_j}}| \},\]
\[ \mathring{V}(\gamma)_i =\{z=[z_0,\cdots,z_t]\in \P_{rig, L}^t, \forall j\leq t, |\frac{z_i}{\varpi^{\gamma_i}}|\geq |\frac{z_j}{\varpi^{\gamma_j}}| \},\]
\[ V_i =\{z=[z_0,\cdots,z_t]\in \P_{rig, L}^t,  z_i\neq 0 \}.\] 
Alors, $X^d_t(\beta)\to \P_{rig, L}^t$ se trivialise sur $\VC (\beta)$, $Y^d_t(\gamma)\to \P_{rig, L}^t$ sur $\mathring{\VC }(\gamma)$, $Z^d_t\to \P_{rig, L}^t$ sur $\VC$ ie.
\[f^{-1} (V(\beta)_i)\cong V(\beta)_i\times \B_L^{d-t}(-\beta_i),\] 
\[f^{-1} ( \mathring{V}(\gamma)_i)\cong  \mathring{V}(\gamma)_i\times  \mathring{\B}_L^{d-t}(-\gamma_i),\]
\[f^{-1} ( V_i)\cong  V_i\times  \A^{d-t}\] par le biais de l'application \[[z_0,\cdots,z_d]\mapsto [z_0,\cdots,z_t]\times (\frac{z_{t+1}}{z_i},\cdots , \frac{z_d}{z_i}).\]
Appelons $\UC=\{U_i\}$ le recouvrement adapté (au cas algébrique, tubulaire ouvert ou fermé) et $F_i$ la fibre sur $U_i$ (soit $\B_L^{d-t}(-\beta_i)$ pour les tubulaires fermés, $\mathring{\B}_L^{d-t}(-\gamma_i)$  pour les tubulaires ouverts, $\A^{d-t}$ en algébrique). La variable sur la base $\P^t_{rig, L}$ sera notée $z=[z_0,\cdots,z_t]$ et celle de la fibre $w=(w_1,\cdots,w_{d-t})$. Sur chaque intersection\footnote{\label{fnotint}
Pour tout recouvrement $\{U_i\}_{i\in S}$ par des ouverts d'un espace rigide $X$ et toute partie finie $I\subset S$, on pourra noter pour simplifier $U_I:=\bigcap_{i\in I}U_i$. 
} $U_{\{i,j\}}$, l'application de transition rend le diagramme suivant commutatif \[
\xymatrix{ 
f^{-1} (U_{\{i,j\}}) \ar[r]^-{\sim} \ar[d]^{\rm Id}  &  U_{\{i,j\}}\times F_i \ar[d]^{{\rm Id}\times m_{ \frac{z_i}{z_j}}}  \\
f^{-1} (U_{\{i,j\}}) \ar[r]^-{\sim} & U_{\{i,j\}}\times F_j }
\] où $m_{\frac{z_i}{z_j}}$ est l'homothétie de rapport $\frac{z_i}{z_j}$. 
On écrira $f^*(\VC (\beta))=\{f^{-1}(V(\beta)_i)\}$, $f^*(\mathring{\VC }(\gamma))=\{f^{-1}(\mathring{V}(\gamma)_i)\}$, $f^*(\VC)=\{f^{-1}(V_i)\}$ les recouvrements de $X^d_t(\beta)$, $Y^d_t(\gamma)$, $Z^d_t$ obtenus.

Dans le cas algébrique, les intersections d'éléments du recouvrement $f^*(\VC)$ sont des produits de copies de $\A^1$ et de $\A^1\backslash\{0\}$. Dans le cas tubulaire fermé, les intersections sur $f^*(\VC (\beta))$ sont des produits de polycouronnes et de polydisques fermés. 

Remarquons que si $t=d$, $X^d_d(\beta)=\P^d_{rig, L}$ , la famille des $(X^d_t(\beta))_{\beta,m,t}$ contient les espaces projectifs. Enfin, il pourra être utile de renormaliser les variables de $ \P_{rig, L}^t$ et de les réécrire sous la forme  $$\tilde{z}_i =\frac{z_i}{\varpi^{\beta_i}}.$$

\begin{ex}\label{excard2}
Pour illustrer les constructions précédentes,  décrivons ici les affinoïdes ${\rm Uni}(\AC)$ lorsque $\AC$ est un arrangement tubulaire fermé  d'ordre $n\geq 1$ avec $|\AC|=2$.   Choisir $\AC$ revient  à se donner deux  hyperplans $H_{a}$ et $H_{b}$  différents  dans $\HC_{n+1}$,  ou par dualité,  deux vecteurs unimodulaires $a$ et $b$ qui n'engendrent pas la même droite dans $(\OC_{K}/\varpi^{n+1})^{d+1}$.  Quitte à réaliser un changement de variables dans $\gln_{d+1}(\OC_K)$ (cf \ref{remgH}),  on peut trouver une base $(e_i)$ de  $\OC_{K}^{d+1}$ tel que\footnote{Notons que $k$ est le plus grand entier $s$ tel que  $a\equiv b \mod \varpi^{s}$ et est par conséquent inférieur strict à $n+1$.  }
\[
\begin{cases}
e_0 = a, \;\; e_1=b & \si a\neq b \mod \varpi \\
e_0= a, \;\; e_0+\varpi^k e_1 =b &   \sinon 
\end{cases}.
\]
Nous allons raisonner  sur ce système de coordonnées. 

Dans le premier cas,  
\begin{eqnarray*}
\mathring{H}_{a}(|\varpi^n|)^{c} & = &  \{ z\in \P^{d}(C) :  \forall   j\leq d,   |z_j| \leq |\varpi^{-n}z_0|   \} \\ 
\mathring{H}_{b}(|\varpi^n|)^{c} & = &  \{ z \in \P^{d}(C) :  \forall   j\leq d,   |z_j| \leq |\varpi^{-n}z_1|   \}. 		
\end{eqnarray*}
Ainsi ${\rm Uni}(\AC)= \{ z \in  \P^{d}(C) : \exists i \leq 1  \mbox{ tel que } \forall j \leq d \;\;  |z_i| \geq  |\varpi^nz_j|   \}= X_{1}^{d}(n,n)$.  
Sous cette presentation,  le recouvrement $V(\beta)$ avec $\beta= (n,n)$ correspond au recouvrement $\{ \mathring{H}_a(|\varpi^n|)^c,  \mathring{H}_b(|\varpi^n|)^c\}$.

Dans le second cas,  
\begin{eqnarray*}
{\rm Uni}(\AC) & = &   \{z \in \P^{d}(C):  \forall j\leq d,  |z_j| \leq |\varpi^{-n}z_0| \mbox{ ou } \forall j\leq d,  |z_j|\leq |\varpi^{-n}(z_0+\varpi^k z_1)| \} \\ 
		& = &  \{z \in \P^{d}(C):  \forall j\leq d,  |z_0| \geq |\varpi^{n}z_j| \mbox{ ou } \forall j\leq d,  | z_1 |\geq |\varpi^{n-k}z_j| \} \\
		& = & X_1^{d}(n,n-k).
\end{eqnarray*}
En particulier,  $V(\beta)$ avec $\beta=(n,n-k)$ est constitué des éléments 
\begin{gather*}
\mathring{H}_a(|\varpi^n|)^c= \{ z\in \P^d(C):  \forall i\leq d,  |{z_i}| \leq |\varpi^{-n}{z_0}|\} \\
\mathring{H}_{e_1}(|\varpi^{k-n}|)^c =\{ z \in  \P^d(C):  \forall i\leq d,  |{z_i}|\leq |\varpi^{n-k}{z_1}| \}. 
\end{gather*}

Notons que les vecteurs $a$ et   $b$ jouent un rôle symétrique. En les échangeant, on obtient cette présentation pour l'union :
\begin{eqnarray*}
{\rm Uni}(\AC) & = &   \{z \in \P^{d}(C):  \forall j\leq d,  |z_0+\varpi^k z_1| \geq |\varpi^{n}z_j| \mbox{ ou } \forall j\leq d,  | z_1 |\geq |\varpi^{n-k}z_j| \} \\
		& = & X_1^{d}(n,n-k).
\end{eqnarray*}
Sous cette présentation, le recouvrement $V(\beta)$ associé est alors constitué des éléments \[\{\mathring{H}_b(|\varpi^n|)^c, \mathring{H}_{e_1}(|\varpi^{n-k}|)^c \}.\]
\end{ex}

\section{Enoncés et stratégies\label{sssectionenoncestrat}}

  Dans les \'enonc\'es ci-dessous nous utiliserons syst\'ematiquement la topologie analytique.
  Nous allons prouver (voir \ref{theoacyco+arralg}, \ref{lemo+constzdt}, \ref{coroorptxdt},  \ref{theoacyco+arrtubfer}, \ref{theocohoanarralggen}):

\begin{theo}\label{theoacyco+}

\begin{enumerate}
\item Les espaces projectifs, les fibrations $Z_t^d$, les arrangements tubulaires ferm\'es et les arrangements alg\'ebriques g\'en\'eralis\'es ${\rm Int}(\AC)$ sont $\Of^{(r)}$-acycliques. 

\item Les sections globales de $\Of^{(r)}$ sur les arrangements alg\'ebriques g\'en\'eralis\'es ${\rm Int}(\AC)$ sont constantes.

\item La cohomologie de $\Of^{(r)}$ sur 
$X^d_t(\beta)$ est concentrée en degrés $0$ et $t$. Quand $t\neq 0$, les sections globales sont constantes et la cohomologie en degré $t$ s'identifie au complété $p$-adique de 
\[\drt{\hzar{t}(\P^t_{zar,\OC_L},\Of(-|\alpha|))\otimes \OC_L^{(r)}}{\alpha\in \N^{d-t}\\|\alpha|\ge t+1}{}.\]

\end{enumerate}
\end{theo}

   Voir \ref{theoacyco+arralg}, \ref{lemo+constzdt},  \ref{coroostst}, \ref{theocohoanarralggen}  pour le r\'esultat suivant: 

\begin{theo}\label{theoacycostst}

\begin{enumerate}
\item Les espaces projectifs, les fibrations $Z_t^d$, les arrangements tubulaires ferm\'es et les arrangements alg\'ebriques g\'en\'eralis\'es ${\rm Int}(\AC)$ sont $\Of^{**}$-acycliques.

\item Les sections globales de $\Of^{**}$ sur les arrangements alg\'ebriques g\'en\'eralis\'es ${\rm Int}(\AC)$ sont constantes.

\item La cohomologie de $\Of^{**}$ sur 
 $X^d_t(\beta)$ est concentrée en degrés $0$ et $t$. Les sections globales sont constantes quand $t\neq 0$.
\end{enumerate}
\end{theo}

  Le r\'esultat suivant est une combinaison de \ref{theocohoangmxdt}, \ref{theocohoangmarrtubfer},   \ref{theoostarrtubfer}, \ref{theocohoanarralggen}. 

\begin{theo}\label{theoacycgm}

\begin{enumerate}
\item Les espaces projectifs vérifient : \[\han{k} (\P^t_{rig,L},\G_m)=\begin{cases}
L^* &{\rm si}\ k=0\\
\Z &{\rm si}\ k=1\\
0 &{\rm sinon}.
\end{cases}
\]
\item  La fibration $f: X^d_t(\beta)\to\P^t_{rig,L}$ induit une décomposition en produit direct pour 
$s>0$:
\[\han{*} (X_t^d (\beta), \G_m)\cong \han{*} (X_t^d (\beta), \Of^{**})\times \han{*} (\P^t_{ rig, L}, \G_m).\]
De plus, les sections globales sont constantes quand $t\neq 0$.
\item  Les arrangements tubulaires fermés ${\rm Int}(\AC)$ sont $\G_m$-acycliques et 
\[\Of^*({\rm Int} (\AC))/L^*\Of^{**}({\rm Int} (\AC))= \Z[\AC]^0.\] 
\item  
Les arrangements algébriques généralisés ${\rm Int}(\AC)$ sont $\G_m$-acycliques et \[\Of^*({\rm Int} (\AC))/L^*=\Z\left\llbracket \AC\right\rrbracket^0.\] 
\end{enumerate}
\end{theo}


Pour obtenir ces résultats, il faut d'abord calculer la cohomologie de $\Of^{(r)}$ sur $X^d_t(\beta)$ (point 3. de \ref{theoacyco+}) grâce  aux résultats d'acyclicité de Van der Put (voir \cite{vdp} ou \ref{theovdp}) et au calcul de la cohomologie de Cech sur le recouvrement $f^*(\VC(\beta))$ (cf \ref{coroorptxdt}). Plus précisément, la fibration $f$ permet de relier le complexe de Cech de $X^d_t(\beta)$ aux complexes de $\P^t_{rig,L}$ pour les faisceaux tordus $\Of^{(r)}(k)$ (point 1. de \ref{theoacyco+} et \ref{lemo+ptrig} pour un énoncé plus fin). Le résultat se déduit de la cohomologie des faisceaux $\Of(k)$ sur les espaces projectifs algébriques sur $\OC_L$. 

Le résultat d'acyclicité pour les arrangements tubulaires fermés découle de l'annulation de la cohomologie de $X^d_t(\beta)$  à partir du degré $t+1$ et de l'argument combinatoire \ref{lemmayervietitere} qui remplace la suite spectrale \eqref{eqsuitspectgen}.

Le transfert des énoncés sur $\Of^{(r)}$ à  $\Of^{**}$ résulte d'un argument sur les logarithmes tronqués \ref{lemoroetun}. Pour le faisceau $\G_m$, on calcule encore la cohomologie de Cech des fibrations $X^d_t(\beta)$ sur le recouvrement $f^*(\VC(\beta))$. Mais on a pour tout $I\subset \llbracket 0,t\rrbracket$ une décomposition \[\Of^* (f^{-1}(V(\beta)_I))=L^* \Of^{**} (f^{-1}(V(\beta)_I))\times \langle \frac{z_i}{z_j}:i,j\in I\rangle_{\Z\modut}\]
qui induit les décompositions  de la cohomologie du point 2. de  \ref{theoacycgm} (cf \ref{theocohoangmxdt}) et celle des sections inversibles dans \ref{theoacycgm} point 3. (cf \ref{coroostst}). Nous notons aussi que le complexe induit par les facteurs directs $\langle\frac{z_i}{z_j}:i,j\in I\rangle_{\Z\modut}$ est celui apparaissant en géométrie algébrique, ce qui permet d'établir le point 1. de \ref{theoacycgm} par comparaison. D'après ce qui précède, on sait que la cohomologie de $X^d_t(\beta)$ s'annule à partir du degré $t+1$ ce qui nous donne l'acyclicité des arrangements tubulaires fermés pour $\G_m$ toujours grâce au lemme combinatoire \ref{lemmayervietitere}.


Pour les arrangements algébriques généralisés $\AC$, ils peuvent être approximés par des arrangements tubulaires fermés compatibles $\AC_n$ d'ordre $n$. On dispose d'une suite exacte pour tout $s>0$
\[0\to \rrr^1\varprojlim_n \han{s-1}({\rm Int} (\AC_{n}),\G_m)\to \han{s} ({\rm Int} (\AC),\G_m)\to \varprojlim_n \han{s}({\rm Int} (\AC_{n}),\G_m)\to 0.\]
Le calcul dans le cas tubulaire fermé induit l'annulation de la cohomologie de $\G_m$ pour  $s>1$ et l'égalité
$\han{1} ({\rm Int} (\AC),\G_m)= \rrr^1\varprojlim_n\Of^*({\rm Int} (\AC_{n}))$.
Il s'agit alors de prouver $\rrr^1\varprojlim_n\Of^*({\rm Int} (\AC_{n}))=0$. D'après la décomposition \ref{theoacycgm} point 3. et le lemme \ref{propr1lim}, il suffit de trouver une constante  $c$ indépendante  de $n$ pour laquelle on a l'inclusion 
\[\Of^{(r)}({\rm Int}(\AC_{n}))\subset \OC_L^{(r)}+\varpi\Of^{(r)}({\rm Int}(\AC_{n-c})).\]
Pour établir cette identité, on raisonne par récurrence sur le rang de ${\rm Int}(\AC_{n})$ et on se ramène à montrer (cf \ref{corodecsimp} et \ref{lemobsarrtubferor} pour voir que cette condition est bien suffisante) que l'image de la flèche
\[\han{{\rm rg}(\AC_n)-1}({\rm Uni}(\AC_n),\Of^{(r)})\to \han{{\rm rg}(\AC_n)-1}({\rm Uni}(\AC_{n-1}),\Of^{(r)}) \]
est contenu dans $\varpi\han{{\rm rg}(\AC_n)-1}({\rm Uni}(\AC_{n-1}),\Of^{(r)})$. Grâce à \ref{theoacyco+} point 1., on peut voir ces groupes de cohomologie comme des sous-groupes des fonctions bornés de polycouronnes (cf \ref{remhtanxdt}) dont les flèches de restriction sont explicites et bien comprises d'après    \ref{lemobspolycourdisqrel}. Le résultat découle alors de ce cas particulier.


\'Etudier la cohomologie des arrangements tubulaires fermés via 
les espaces $X^d_t (\beta )$ est semblable à la stratégie de \cite{scst}. 
Par exemple, le point 3. de \ref{theoacyco+} mimique l'axiome d'homotopie de \cite[§2]{scst}. S'intéresser à $\Of^{(r)}$ puis à $\Of^{**}$ et enfin à $\G_m$ rappelle la preuve de \cite[théorème 3.25]{vdp}. L'argument de passage à la limite s'inspire de \cite[sous-section 1.2]{GPW1}.

\begin{rem}\label{remgenprinc}

\begin{itemize}
\item Tous les calculs sur la cohomologie de Cech qui apparaissent dans la preuve de ces résultats peuvent être réalisés quand $K$ est de caractéristiques $p$. Cette observation suggère que ces résultats peuvent aussi être vérifiées dans ce cadre. Toutefois, nous ne voyons pas comment adapter la preuve du lemme cruciale \ref{lemoroetun} dans ce cas. En particulier, nous ne sommes pas en mesure d'établir un analogue du théorème \ref{theoacycostst} ainsi que de sa conséquence théorème \ref{theoacycgm} en caractéristique $p$.

\item Nous avons choisi de ne pas étudier les arrangements tubulaires ouverts même si certains raisonnements semblent pouvoir être adaptés. En fait, ces arrangements peuvent s'écrire comme des unions croissantes d'arrangements tubulaires fermés où l'on s'est autorisé des ordres rationnels. Si l'on pouvait établir des résultats similaires pour ces généralisations (la combinatoire des unions est similaire dans ce cadre), on pourrait alors grâce à la suite exacte montrer que la cohomologie des faisceaux étudiés est concentrés en degrés $0$ et $1$ (grâce à \eqref{eq:suitexrlimh}). Un travail futur pourrait chercher à savoir si on a encore l'acyclicité pour les intersections  des arrangements tubulaires ouverts lorsque le corps $L$ est sphériquement clos.
\end{itemize}
\end{rem}
    Tous nos calculs utilisent de manière cruciale le résultat suivant de van der Put \cite[th. 3.10, th. 3.15, th. 3.25]{vdp}, décrivant la cohomologie de quelques affinoïdes simples. 
    
\begin{theo}[Van der Put]\label{theovdp}

Les produits de polycouronnes et polydisques fermés\footnote{Plus généralement les polydisques généralisés au sens de \cite[3.9]{vdp}} n'ont pas de cohomologie analytique en degré strictement positif pour:

\begin{enumerate}
\item les faisceaux constants 
\item le faisceaux  $\Of^{(r)}$ 
\item le faisceau $\Of^+$ en dimension $1$ 
\item le faisceau  $\G_m$ 
\end{enumerate}

\end{theo}

\begin{rem} Un théorème de Bartenwerfer \cite{bart} affirme que les boules fermées sont aussi acycliques pour le faisceau $\Of^+$, en toute dimension. Nous ne savons pas si ce résultat est encore vrai pour les couronnes  (sauf en dimension $1$, comme indiqué). Si c'était le cas, beaucoup des résultats à suivre 
pourraient aussi être énoncés pour $\Of^+$.

\end{rem}

\section{Cas des arrangements algébriques  \label{paragrapho+arralg}}

Nous traitons d'abord le cas des arrangements alg\'ebriques. L'énoncé suivant est un analogue de \ref{lemobsarrtubferadd} dans le cas particulier des polycouronnes où le résultat est direct. Pour la généralisation \ref{lemobsarrtubferadd}, nous nous ramenons à ce cas particulier grâce au point technique \ref{theocohoano+xdt}. Une fois ce résultat établi, les méthodes dans le cas algébrique sont relativement similaires au cas algébrique généralisé. 
\begin{lem}\label{lemobspolycourdisq}
On considère le produit de polycouronnes et de polydisques suivant \[U=\{x=(x_1 ,\cdots, x_d)\in \A^d_{rig, L} : \forall i, |\varpi|^{-r_i} \ge|x_i|\ge|\varpi|^{s_i}\}\] où $(r_i)_i$ et $(s_i)_i$ sont des entiers\footnote{On s'autorisera $s_i=\infty$ pour les facteurs isomorphes à  une boule fermée.}. De même, on considère  \[V=\{x=(x_1 ,\cdots, x_d)\in \A^d_{rig, L} : \forall i, |\varpi|^{-r_i-1} \ge|x_i|\ge|\varpi|^{s_i+1}\}.\] 
Alors, on a \[\Of^{+}(V)\subset \OC_L+\varpi\Of^{+}(U)\]



\end{lem}

\begin{proof}
La description des espaces $U$ et $V$ nous fournit un système de coordonnées commun $(X_i)_i$. La famille de mon\^omes $(\prod_{i :\nu_i\ge 0} (\varpi^{r_i} X_i)^{\nu_i} \prod_{j : \nu_j< 0} (\frac{\varpi^{s_j}}{X_j})^{ -\nu_j})_{\nu \in E}$ forme une base de Banach \footnote{Rappelons ici cette notion. Pour cela, considérons $A$ une $L$-algèbre (ou une $\OC_L$-algèbre) normée complète et $\ell_{\infty}^0(A)$ l'ensemble des suites à valeurs dans $A$ dont le terme général tends vers $0$ muni de la norme $\left\Vert (a_n)_n\right\Vert_\infty := \max_n |a_n|$. Une base de Banach d'un $A$-module normé complet $M$ est une famille $(e_n)_n\in M^{\N}$ telle que l'application $(a_n)_n\in\ell_{\infty}^0(A)\mapsto\sum a_n e_n \in M$ est bien définie et réalise une isométrie entre les espaces $\ell_{\infty}^0(A)$ et $M$.} de $\Of(U)$ avec 

\[ E:=\{\nu\in \Z^d:\nu_j\ge 0 \text{ si } s_j =\infty \}.\] Il en est de m\^eme pour la famille $(\prod_{i :\nu_i\ge 0} (\varpi^{r_i+1} X_i)^{\nu_i} \prod_{j: \nu_j\le 0} (\frac{\varpi^{s_j+1}}{X_j})^{ \nu_j})_{\nu \in E}$ sur $\Of(V)$. Mais on remarque que pour tout $\nu$ (avec $|\nu|=\sum_{j} \nu_j$)
\[ \prod_{i :\nu_i\ge 0} (\varpi^{r_i+1} X_i)^{\nu_i} \prod_{j: \nu_j\le 0} (\frac{\varpi^{s_j+1}}{X_j})^{ \nu_j}= \varpi^{| \nu|} \prod_{i :\nu_i\ge 0} (\varpi^{r_i} X_i)^{\nu_i} \prod_{j: \nu_j\le 0} (\frac{\varpi^{s_j}}{X_j})^{\nu_j}. \]
Il est alors ais\'e de voir que si une section \`a puissance born\'ee de $V$ n'a pas de terme constant, sa restriction est dans $\varpi \Of^+(U)$. 

\end{proof}

Nous avons une version relative de ce résultat 

\begin{lem}\label{lemobspolycourdisqrel}
Soit $Y$ un affinoïde sur $L$ et $U$,$V$ les affinoïdes définis dans le lemme précédent, alors on a
\[\Of^+(Y\times V)\subset \Of^+(Y)+\varpi \Of^+(Y\times U).\]
\end{lem}

\begin{proof}
C'est le même argument que pour \ref{lemobspolycourdisq} et cela s'obtient en comparant les développements uniques en série sur les deux espaces $U$,$V$ :
\begin{lem}\label{lemcourrel}

Soit $Y=\spg(A)$ un affinoïde réduit sur $L$ et $U$ comme précédemment. Toute section de $Y\times U$ admet une écriture unique
\[\sum_{\nu}f_\nu Z^\nu\ {\rm avec}\ f_{\nu}\in \Of(Y)\ {\rm et}\ \left\Vert f_\nu\right\Vert_Y \left\Vert Z^\nu\right\Vert_U\to 0\]
où la variable $\nu$ parcourt l'ensemble des vecteurs\footnote{ie. $\nu\in E$ en reprenant les notations de la preuve de \ref{lemobspolycourdisq}} de $\Z^d$ tel que $\nu_i\ge 0$ quand $s_i=\infty$. De plus, la norme spectrale vérifie l'identité $\left\Vert \sum_{\nu}f_\nu Z^\nu\right\Vert_{Y\times U}=\max_\nu\left\Vert f_\nu\right\Vert_Y\left\Vert Z^\nu\right\Vert_U$
\end{lem}

\begin{proof}

On commence par établir deux identités classiques sur $\ell_{\infty}^0 (\OC_L)$. Donnons-nous une $\OC_L$-algèbre plate normée complète, les flèches naturelles\footnote{\label{footop} Décrivons la topologie sur le module $B[\frac{1}{\varpi}]$. Étant donné une norme sur $B$ définissant la topologie sur cette algèbre, cette dernière peut être prolongée de manière unique en une norme sur $B[\frac{1}{\varpi}]$ qui vérifie la relation $\left\Vert b\right\Vert_{B[\frac{1}{\varpi}]}=\left| \varpi^{-k}\right| \left\Vert b \varpi^k\right\Vert_B$ pour $b\in \varpi^{-k}B\subset B[\frac{1}{\varpi}]$ (ne dépend pas de $k$). Cela permet de définir la topologie sur $B[\frac{1}{\varpi}]$ où toute base de voisinage de $0$ dans $B$ définit aussi une base de voisinage de $0$ dans $B[\frac{1}{\varpi}]$. La complétude de $B[\frac{1}{\varpi}]$ se déduit alors de celle de $B$.} $\ell_{\infty}^0 (\OC_L)\to \ell_{\infty}^0(B)$ et $\ell_{\infty}^0(B)\to \ell_{\infty}^0(B[\frac{1}{\varpi}])$ induisent des isomorphismes  :
\begin{equation}\label{eq:ell0}
\ell_{\infty}^0 (\OC_L)\hat{\otimes} B\cong \ell_{\infty}^0(B) \et \ell_{\infty}^0(B[\frac{1}{\varpi}])\cong \ell_{\infty}^0(B)[\frac{1}{\varpi}].
\end{equation}
Nous commençons par la deuxième. Vu comme des sous-groupes de $B[\frac{1}{\varpi}]^\N$, on a une inclusion évidente entre les modules $\ell_{\infty}^0(B)[\frac{1}{\varpi}]\subset \ell_{\infty}^0(B[\frac{1}{\varpi}])$. Prouvons celle en sens opposé. Une suite $(u_n)_n$ dont le terme général tend vers $0$ dans $B[\frac{1}{\varpi}]$ est à valeurs dans $B$ à partir d'un certain rang $N_0$. En particulier, on peut trouver un entier $k$ assez grand tel que $\varpi^k u_n \in B$ pour $n\le N_0$ et  $(\varpi^ku_n)_n\in \ell_{\infty}^0(B)$ ce qui prouve l'inclusion voulue.

Intéressons-nous maintenant à la première identité. On a une flèche naturelle $\ell_{\infty}^0 (\OC_L){\otimes} B\to \ell_{\infty}^0 (B)$ et nous voulons montrer que c'est un isomorphisme lorsqu'on complète. Les suites à support fini $\ell^c (\OC_L)\subset  \ell_{\infty}^0 (\OC_L)$  et $\ell^c (B)\subset  \ell_{\infty}^0 (B)$ forment des sous-groupes denses  qui vérifient $\ell^c(B)\cong \ell^c (\OC_L){\otimes} B $ et cette isomorphisme s'étend par continuité. 

Revenons à la situation de l'énoncé. Par hypothèse de réduction sur $Y$, $A^+$ est un anneau de définition\footnote{ie.  $A^+$ est une $\OC_L$-algèbre topologique plate et $p$-adiquement complète tel que $A$ est homéomorphe à $A^+[1/\varpi]$(voir \eqref{footop} pour la topologie sur $A^+[1/\varpi]$).} de $A$. Ainsi, on a
\begin{eqnarray*}
\Of(Y\times U) & = &  (A^+ \widehat{\otimes}_{\OC_{L}} \Of^+(U) )[\frac{1}{p}] \\
\end{eqnarray*}
Les identités précédentes \eqref{eq:ell0} montrent qu'une base de Banach de $\Of^+(U)$ définit aussi une base Banach  sur $\Of(Y\times U)$. On en déduit l'existence et l'unicité de l'écriture en somme voulue. 

Prouvons l'égalité pour la norme spectrale. D'après la discussion précédente, développons une section sous la forme \footnote{On rappelle l'égalité $\left\Vert Z^\nu\right\Vert=\prod_{i :\nu_i\ge 0} \varpi^{\nu_i r_i} \prod_{j : \nu_j< 0} \varpi^{ -\nu_j s_j}$ pour $\nu\in E$.} $f=\sum_{\nu}f_\nu \frac{Z^\nu}{\left\Vert Z^\nu\right\Vert}$ et appelons $\pi :Y\times U\to Y$ la projection. Pour tout $y\in Y(C)$, 
la norme spectrale sur $\pi^{-1}(y)$ est donnée par $\max_\nu(|f_\nu(y)|)$ (voir le cas d'un corps). La norme spectrale totale  vérifie $\left\Vert f\right\Vert_{Y\times U}=\max_y \left\Vert f\right\Vert_{\pi^{-1}(y)}=\max_y \max_\nu(|f_\nu(y)|)=\max_\nu(|f_\nu|)$ et on en déduit l'égalité voulue.

\end{proof}

\end{proof}

Nous avons aussi un résultat similaire pour les fonctions inversibles des couronnes relatives.

\begin{lem}\label{claimostrelpolicd}
Soient $I \subset \left\llbracket 1,n\right\rrbracket$, $(s_i)_{i\in I}$, $(r_i)_{i\in I}$ des nombres rationnels tels que $s_i\ge r_i$ pour tout $i$, $\mathrm{Sp}(A)$ un $L$-affino\"ide réduit et connexe, et soit $D$ la polycouronne 
\[ \{ (x_1, \dots , x_n) \in \B_L^n \; | \; |\varpi|^{s_i} \leq|x_i| \leq |\varpi|^{r_i} \text{ si } i \in I \}. \]
Alors \[ \Of^{*}(D \times \mathrm{Sp}(A)) = \Of^*( \mathrm{Sp}(A)) \Of^{**}(D \times \mathrm{Sp}(A)) \times \left\langle x_i: \; i \in I\right\rangle_{\Z\modut}. \]
Le résultat reste vrai si la polycouronne $D$ est ouverte.
\end{lem}

\begin{proof}

Supposons vrai le cas d'une couronne fermée et montrons le résultat dans le cas d'une couronne ouverte.    On se donne un recouvrement croissant $D=\bigcup_n D_n$ par des couronnes fermées,  et on note pour simplifier $D_A=\spg A \times D$ (idem  pour $D_{n,A}$).  On a alors par hypothèse:
\[\Of^* (D_A)=\bigcap_n \Of^* (D_{n,A})=\left( \bigcap_n A^*\Of^{**}(D_{n,A})\right)\times  \left\langle x_i: \; i \in I\right\rangle_{\Z\modut}. \]
 Il s'agit d'établir $\bigcap_n A^*\Of^{**}(D_{n,A})=A^*\Of^{**}(D_A)$.


Prenons $u$ dans cette intersection et écrivons $u=\lambda_n(1+h_n)$ dans chaque $A^*\Of^{**}(D_{n,A})$. Fixons $n_0\in\N$, pour tout $n>n_0$, on observe
\[\frac{\lambda_n}{\lambda_{n_0}}=\frac{1+h_{n_0}}{1+h_n}\in \Of^{**}(D_{n_0,A})\cap A^*=A^{**}, \]
donc $\frac{u}{\lambda_{n_0}}=\frac{\lambda_n}{\lambda_{n_0}}(1+h_n)\in \bigcap_n \Of^{**}(D_{n,A})=\Of^{**}(D_A)$ et ainsi $u\in \lambda_{n_0}\Of^{**}(D_A)$. L'autre inclusion étant claire, on en déduit le résultat pour les couronnes ouvertes.  

On suppose maintenant la couronne $D$  fermée.  Par récurrence sur $n=\dim D$,  on se ramène au cas $n=1$ et à  la distinction $I=\{1\}$ ou $I=\emptyset$.   En effet, en dimension supérieur, $D$ se décompose comme un produit de couronnes fermées  $D=D' \times D''$ obtenu en projetant sur la dernière coordonnée pour $D''$,   et  sur  les autres pour $D'$.   On a par hypothèse de récurrence sur $D''$:
\[
\Of^*(D_A)= \Of^*(D'_A) \Of^{**}(D_A)  \times \langle x_{i} \; : \; i \in I \cap \{n\}  \rangle_{\Z \modut}. 
\]
Par hypothèse de récurrence pour $D'$ on a   $\Of^*(D'_A)= A^* \Of^{**}(D'_A) \times \langle   x_i \; : \;  i\in I\et i< n  \rangle_{\Z \modut}$.  On en déduit le résultat voulu.

 Supposons $\dim D=1$ et notons $x$ la variable sur $D$. Commençons par le cas où $A$ est  une extension complète du corps $L$. Étant donné une fonction $f$ inversible sur $D_A$, on peut trouver, d'après \cite[2.2.4]{frvdp}, une fraction rationnelle $g$ n'ayant aucun pôle ni zéro sur $D_A(\hat{\bar{A}})\subset \hat{\bar{A}}$ telle que $f/g\in \Of^{**}(D_A)$. Il est alors suffisant de prouver l'existence d'un entier relatif $k$ et d'une constante $\lambda\in A^*$ telle que $\frac{g}{\lambda x^k}\Of^{**}(D_A)$. Écrivons $g$ en un produit de monôme $c\prod_{m\in M} (x-m)^{\alpha_m}$ avec $ M\subset \hat{\bar{A}}$, et décomposons l'ensemble en $M=M^+\amalg M^-$ avec \[M^+=\{m\in M: |m|> |\varpi|^{r_1}\}\et M^-=\{m\in M: |m|< |\varpi|^{s_1}\}.\] Comme le groupe de Galois  absolu de $A$ agit par isométrie sur $\hat{\bar{A}}$, on a $\prod_{m\in M^+} m^{\alpha_m}\in A$. De plus, on observe les relations 
 \[(x-m)/(-m)=1-x/m\in \Of^{**}(D_{\hat{\bar{A}}}) \text{ si } m\in M^+,\]
 \[(x-m)/x=1-m/x\in \Of^{**}(D_{\hat{\bar{A}}}) \text{ sinon } m\in M^-.\]
On pose alors $k:=\sum_{m\in M^-} \alpha_m$ et $c\prod_{m\in M^+} (-m)^{\alpha_m}\in A$ de telle manière que \[ \frac{g}{\lambda x^k}\in \Of^{**}(D_{\hat{\bar{A}}})\cap \Of(D_A)=\Of^{**}(D_A)\] ce qui établit le cas d'un corps.

Revenons au cas général et montrons qu'il découle du cas particulier des corps.      Soit $u$ une section inversible de $D_A$,  alors pour tout $z\in \spg(A)$ on a une décomposition \begin{equation}
\label{eqDecomposition} u(z)=\lambda_z (1+h_z) x^{\beta_{z}}\in \Of^* (\spg(K(z))\times D)=\Of ^*(D_z)\end{equation} avec $\lambda_z \in K(z)^*$, $h_z\in \Of^{++}(D_z)$ et $\beta_{z}\in \Z$. 
   Si $I=\emptyset$, on a $\beta_{z}=0$ pour tout $z$.  Sinon, nous montrons que la fonction $z\mapsto\beta_{z}$ est continue sur $\spg A$,  d'où localement constante.   Soit  $z_0\in \spg A$ fixé, quitte à  multiplier $u$ par $x^{-\beta_{z_0}}$, on peut supposer $\beta_{z_0}=0$.  On écrit $u$ comme une somme grâce à \ref{lemcourrel}
\[
a_0+\sum_{\nu> 0} a_{\nu} \left(\frac{x}{\varpi^{r}}\right)^{\nu}+ \sum_{\nu>0} a_{-\nu} \left( \frac{\varpi^{s}}{x}\right)^{\nu} =a_0+\tilde{u}
\]
avec  $a_{\nu}\to 0$ pour le filtre des parties finies.    Si $I=\emptyset$, on a $a_\nu=0$ si $\nu<0$.  Notons que la décomposition $\sum_{\nu} a_{\nu}(z_{0})x^{\nu}= \lambda_{z_0}(1+ h_{z_0}) $ entraîne 
\begin{equation}\label{eq:decost}
a_{0}(z_0) \in \lambda_{z_0} K(z_0)^{**} \et a_{\nu}(z_0)\in \lambda_{z_0}K(z_0)^{++}
\end{equation}
 pour tout $\nu \neq 0$.   On peut trouver un voisinage affinoide $U$ de $z_0 $ dans $\spg A$ où $\lambda_{z_0}$ se relève en un élément inversible $\tilde{\lambda}\in \Of^*(U)$.        Soit $N>0$ tel que $a_{\nu} \in  \tilde{\lambda} \Of^{++}(U)$ pour tout $|\nu|>N$ et on fixe $\epsilon<1$ dans $p^{\Q}$ tel que $\left| \frac{a_0(z_0)}{\lambda_{z_0}}-1\right| \leq \epsilon$ et $\left| \frac{a_{\nu}(z_0)}{\lambda_{z_0}} \right|\leq \epsilon$ pour tout $|\nu| \leq N$.  Considérons l'ouvert affinoide $V$ de $U$ donné par 
\[V:=\{z\in U : |\frac{a_0(z)}{\tilde{\lambda}}-1|\leq \varepsilon \et |\frac{a_{\nu}(z)}{\tilde{\lambda}}|\leq \varepsilon \ \ \  \forall |\nu|\leq N\}.\] 
  Alors $z_0\in V$ pour $\varepsilon$ assez proche de $1$ d'après \eqref{eq:decost}. On a dans ce cas  $a_{0}\in  \Of^*(V)$ avec $|a_0|_{V}=|\tilde{\lambda}|$ et $\frac{a_{\nu}}{\tilde{\lambda}},\frac{a_{\nu}}{a_0}$ sont dans $\Of^{++}(V)$ pour tout $\nu \in \Z$.     On en déduit que la restriction de $u$ à la couronne  $D_{V}=V\times D$ s'écrit sous la forme $u= a_0(1+\frac{\tilde{u}}{a_0})$ avec $a_0 \in \Of^*(V)$ et $\frac{\tilde{u}}{a_0}\in \Of^{++}(D_V)$.  Cela montre  par unicité de $\beta_z$ dans (\ref{eqDecomposition}),  que  pour tout $z\in V$  on a $\beta_{z}=0$ comme voulu.  Ainsi,  $z\mapsto \beta_{z}$ est  constante sur les ouverts d'un recouvrement admissible, et donc constante par connexité de $\spg A$.

Supposons maintenant que $\beta_{z}=0$ pour tout $z\in \spg A$.  L'argument précédent montre que pour tout $z\in \spg A$ on a $a_0(z)\neq 0$ et $\frac{\tilde{u}}{a_0(z)}\in \Of^{++}(D_z)$.   Donc,  on a $a_0 \in A^*$ et $\frac{\tilde{u}}{a_0}\in \Of^{++}(D_{A})$,  ce qui donne la  décomposition voulue 
\[
u= a_0(1+\frac{\tilde{u}}{a_0})\in A^* \Of^{**}(D_{A}). 
\]
\end{proof}

Le résultat intermédiaire \ref{lemobspolycourdisqrel} est utile au vu du point technique général suivant: 

\begin{prop}\label{propr1lim}

Soit $X=\bigcup_n U_n=\bigcup_n \spg(A_n)$ une réunion croissante de $L$–affinoïdes. Supposons l'existence d'une constante  $c$ indépendante de $n$ tel que

\begin{equation}\label{eq:reso+}
\Of^+(U_{n+c})\subset \OC_L+\varpi\Of^+(U_n).
\end{equation}

Alors les sections globales  des faisceaux $\Of^+$, $\Of^{(r)}$, $\Of^{**}$ et $L^*\Of^{**}$  sont constantes et on a 
\[\rrr^1 \limp_n \Ff(U_n)=0\]
pour $\Ff=\Of^+,\Of^{(r)},\Of^{**},L^*\Of^{**}$.

\end{prop}

Avant de montrer ce résultat clé, commençons par quelques commentaires sur les hypothèses de l'énoncé et sur les foncteurs dérivés de la limite projective. La preuve sera en fait une application du lemme plus général \ref{lemgenrlim}.

Il sera utile d'observer que les conclusions de la proposition seront encore vrais quand on remplace le faisceau $\Of^{+}$ par $\Of^{++}$ et $\OC_L$ par $\mG_L$ dans l'équation \eqref{eq:reso+}. En fait cela découle de l'observation suivante :

\begin{prop}\label{proprlimo+o++}

Soit $X=\bigcup_n U_n=\bigcup_n \spg(A_n)$ une réunion croissante de $L$–affinoïdes. Soit $c$ un entier, $0< r$, on a les implications suivantes :
\[\forall n> 0,\Of^+(U_{n+c})\subset \OC_L+\varpi\Of^+(U_n)\Rightarrow \forall n> 0,\Of^{(r)}(U_{n+2c})\subset \OC^{(r)}_L+\varpi\Of^{(r)}(U_n),\]
\[\forall n> 0,\Of^{(r)}(U_{n+c})\subset \OC_L^{(r)}+\varpi\Of^{(r)}(U_n)\Rightarrow \forall n> 0,\Of^+(U_{n+2c})\subset \OC_L+\varpi\Of^+(U_n).\]
\end{prop}

\begin{proof}
Les  preuves des deux implications sont quasiment identiques et nous ne traiterons que la première. De plus, quitte à multiplier  par une puissance de $\varpi$, on peut supposer que $|\varpi|< r\le 1$. 

Prenons  $f$ dans $\Of^{(r)}(U_{n+2c})$ et donc dans $\Of^{+}(U_{n+2c})$ par hypothèses sur $r$. En appliquant deux  fois l'hypothèse, on a la chaîne d'inclusion
\[\Of^{+}(U_{n+2c})\subset \OC_L+\varpi\Of^{+}(U_{n+c})\subset \OC_L+\varpi^2\Of^{+}(U_{n})\subset \OC_L+\varpi\Of^{(r)}(U_{n}).\] Ainsi, la fonction $f$ s'écrit $f=\lambda+\varpi\tilde{f}$ avec $\tilde{f}\in \Of^{(r)}(U_{n})$ et $\lambda\in \OC_L$. En particulier, $\lambda=f-\varpi\tilde{f}\in \OC_L \cap \Of^{(r)}(U_{n})=\OC_L^{(r)}$ et on en déduit la décomposition voulue.
\end{proof}

Décrivons les quelques propriétés des foncteurs dérivés de la limite projective que nous allons utiliser. Soit  $A$ un anneau\footnote{Au vu de la propriété d'invariance décrite dans \cite[Remarque 1.10.]{jen}), on pourra toujours supposer $A=\Z$}, $I$ un ensemble ordonné filtré, la catégorie des systèmes projectifs en $A$-modules indexés par $I$ est abélienne et possède suffisamment d'objets injectifs (cf \cite[Paragraphe 1]{jen}), et le foncteur « limite projective » admet des foncteurs dérivés à droite que l'on notera $\rrr^i \limp_{j\in I}$.

Dans toute la suite, nous n'étudierons que des systèmes projectifs sur $I=\N$. Un des résultats les plus importants dans ce cas est l'annulation de la plupart des foncteurs dérivés (\cite[Théorème 2.2]{jen}) :
\begin{equation}\label{eq:rilim}
\forall i\ge 2, \rrr^i \limp_{n}=0.
\end{equation}
On peut de plus calculer le premier foncteur dérivé (d'après \cite[Remarque après le Théorème 2.2]{jen}) et ce dernier s'inscrit dans une suite exacte :
\begin{equation}\label{eq:exprrlim}
0\to \limp_n M_n \to \prod_n M_n \fln{\delta}{} \prod_n M_n \to \rrr^1 \limp_n M_n \to 0.
\end{equation}
avec $\delta((m_n)_n)=(m_n - \varphi_{n+1}(m_{n+1}))_n $ où  $\varphi_{n+1} :M_{n+1}\to M_n $ est une des fonctions de transition du système projectif. Par abus, les éléments de $\prod_n M_n$ seront appelés cocycles et ceux dans  ${\rm Im} \delta$ seront des cobords.

Précisons la situation dans laquelle nous allons appliquer ces résultats. Prenons $X$ un espace rigide et $\UC=\{U_n\}$ un recouvrement admissible croissant par des ouverts affinoides et $\Ff$ un faisceau sur $X$. En explicitant le complexe de Cech sur ce recouvrement, on obtient des identifications grâce à la suite exacte \eqref{eq:exprrlim}
\[\hcech{0}(X,\UC,\Ff)\cong \limp_n \Ff(U_n) \et \hcech{1}(X,\UC,\Ff)\cong \rrr^1\limp_n \Ff(U_n). \]
On peut aussi exprimer la cohomologie de $X$ en fonction de celle des ouverts $U_n$. Plus précisément, la composition des foncteurs $\Gamma$, $\limp_n$ nous fournit une suite spectrale 
\[E_2^{i,j}=\rrr^i \limp_n \han{j}(U_n,\Ff_{|U_n})\Longrightarrow \han{i+j}(X,\Ff)\]
qui dégénère d'après le résultat d'annulation \eqref{eq:rilim}. On obtient une suite exacte pour tout $s$ (avec pour convention $\han{-1}(U_n,\Ff_{|U_n})=0$)
\begin{equation}\label{eq:suitexrlimh}
0\to \rrr^1 \limp_n \han{s-1}(U_n,\Ff_{|U_n})\to \han{s}(X,\Ff) \to  \limp_n \han{s }(U_n,\Ff_{|U_n})\to  0.
\end{equation}
  Comme dans le raisonnement précédent, on peut encore interpréter cette suite exacte comme résultant de la dégénérescence de la suite spectrale de Cech sur le recouvrement $\UC$ grâce à \eqref{eq:rilim}.

Nous allons maintenant prouver l'annulation de $\rrr^1 \limp_n M_n$ pour des systèmes projectifs $(M_n)_n$ particuliers.

\begin{lem}\label{lemgenrlim}

Soit une suite décroissante de groupes abéliens complets $(G_n)_n$ dont la topologie est induite par des bases de voisinage par des sous groupes ouverts $(G^{(i)}_n)_i$ avec $G_n=G^{(0)}_n$. 
Supposons  $G^{(i)}_{n+1}\subset G^{(i)}_n$ pour tous $i,n$ (en particulier les inclusions sont continues).

S'il existe un sous-groupe $H\subset \bigcap_n G_n$ fermé dans chaque $G_n$(ie. $H= \bigcap_i H+ G^{(i)}_n$ pour tout $n$) vérifiant
\begin{equation}\label{eqres}
G^{(i)}_{n+c}\subset H+ G^{(i+1)}_n.
\end{equation}
pour une constante  $c$ indépendante de $i,n$ alors  \[\bigcap_n G_n=\limp_n G_n=H \et \rrr^1 \limp_n G_n=0.\]

\end{lem}

\begin{proof}

On veut déterminer $\limp_n G_n$ et donc établir l'inclusion $\bigcap_n G_n\subset  H$ (l'autre étant vérifiée par hypothèse). D'après \eqref{eqres}, on vérifie aisément par récurrence l'inclusion  $G_{cn}\subset H+ G^{(n)}_0$ pour tout $n$  d'où
\[\bigcap_n G_{cn}\subset \bigcap_n H+ G^{(n)}_0=H\]
par hypothèse de fermeture de $H$.

Calculons maintenant le groupe $\rrr^1 \limp_n G_n$. Prenons un cocycle  $(f_n)_n$ et montrons que c'est un cobord. Toujours d'après \eqref{eqres}, on peut trouver par récurrence un suite\footnote{Il suffit de montrer l'inclusion $G^{(i)}_{n+kc+r}\subset G^{(i+k)}_{n} + H$. Comme on a l'inclusion $G^{(i+k)}_{n+r}\subset G^{(i+k)}_{n}$, on peut supposer $r=0$. Quand $k=1$, le résultat est exactement l'hypothèse \eqref{eqres}. Supposons pour tout $i$, $n$, le résultat vrai pour un entier $k$ fixé, on a une chaîne d'inclusion qui termine l'argument \[G^{(i)}_{n+(k+1)c+r}\subset G^{(i+1)}_{n+kc+r}+H\subset G^{(i+k+1)}_{n} + H\] (par \eqref{eqres}  pour la première et par hypothèse de récurrence pour la seconde). } $(h_n)_{n}\in H^{\N}$ tel que pour tout $n$, $k$ et  $r<c$, on a
\[f_{n +kc +r}-h_{n +kc +r}\in  G^{(k)}_{n }.\]
Dans ce cas, la somme $\sum_{m\ge n} f_m-h_m$ converge dans $G_n$ pour tout entier $n$ et vérifie \[\delta((\sum_{m\ge n} f_m-h_m)_n)=(f_n)_n-(h_n)_n.\]   Donnons-nous $\tilde{h}_0\in H$ et construisons par récurrence une suite $(\tilde{h}_n)_n$ tq $\tilde{h}_{n+1}=\tilde{h}_n-h_n$ ie. $\delta ((\tilde{h}_n)_n)=(h_n)_n$. On en déduit que  $(f_n)_n$ est en fait le cobord $\delta((\sum_{m\ge n} f_m-h_m)_n+(\tilde{h}_n)_n)$.

\end{proof}

\begin{proof}\eqref{propr1lim}
Les constantes $\Ff(U_n)\cap L $ forment des fermés de $\Ff(U_n) $ pour  $\Ff=\Of^+,\Of^{(r)},\Of^{**},L^*\Of^{**}$.
Les suites décroissantes  $(\Of^{+} (U_n))_n$ et $(\Of^{(r)}(U_n)_n)$ de groupes topologiques vérifient clairement l'inclusion \eqref{eqres} et \ref{proprlimo+o++} par hypothèse. Montrons que c'est encore le cas pour les suites $(L^*\Of^{**} (U_n))_n$ et $(\Of^{**}(U_n))_n$. Raisonnons uniquement pour le second, le premier s'en déduira aisément. Soit $1+\varpi^k  f$ avec $f\in \Of^{++}(U_n)$, on peut trouver une constante $\lambda\in\mG_L$ telle que $f-\lambda\in \varpi\Of^{++}(U_{n-1})$.   Alors, on a        \[\frac{1+\varpi^k f }{1+\varpi^k \lambda}=1+\varpi^k \frac{f-\lambda}{1+\varpi^k \lambda}\in  1+\varpi^{k+1} \Of^{++}(U_{n-1}). \]          
Le résultat \ref{propr1lim} est alors une conséquence directe du lemme précédent.
\end{proof}


La base canonique $(e_i)_{0\le i\le d}$ de $K^{d+1}$ définit une collection de $d+1$ hyperplans $V^+(z_i)\subset \P^d_{rig, L}$ et on note $\BC$ l'arrangement algébrique $\{V^+(z_i)\}_{0\le i\le r}$.
\begin{coro}\label{coroo+arralg}
 L'espace ${\rm Int} ( \BC)$ défini plus haut est acyclique pour les faisceaux $\Of^{(r)}$, $\Of^{**}$ et $\G_m$. Les  sections globales de $\Of^+$ et $\Of^{**}$ sont constantes et \[\OC^*({\rm Int} ( \BC))=L^*\times T\] avec $T=\left\langle \frac{z_i}{z_0} : 1\le i\le r \right\rangle_{\Z\modut}$
\end{coro}

\begin{proof}
On voit cet arrangement d'hyperplans comme le produit $(\A_{rig, L}^1\backslash\{0\})^r\times\A_{rig, L}^{d-r}$ 
On le recouvre par $(X_n)_n$ où (en posant $x_i=z_i/z_0$)\[X_n=\{x=(x_1 ,\cdots, x_d)\in \A^d_{rig, L} : \forall i\le r, |\varpi|^{-n} \ge|
x_i|\ge|\varpi|^{n},\forall j\ge r+1, |\varpi|^{-n} \ge|x_j|\}.\] On a la suite exacte :
\[ 0 \to \rrr^1 \varprojlim \han{s-1}(X_n, \Ff) \to \han{s}({\rm Int} ( \BC), \Ff) \to \varprojlim \han{s}(X_n, \Ff) \to 0. \] Mais ${\rm Int} ( \BC)=\bigcup_n X_n$  est un recouvrement admissible constitu\'e de produits de polycouronnes et polydisques, chacun des termes est acyclique pour les faisceaux $\Of^{(r)}$, $\Of^{**}$ et $\G_m$ d'apr\`es \ref{theovdp} d'où pour $s>0$ \[\varprojlim \han{s}(X_n, \Ff) =0\et \han{s}({\rm Int} ( \BC), \Ff)=\rrr^1 \varprojlim \han{s-1}(X_n, \Ff).\] Grâce au résultat \ref{lemobspolycourdisqrel}, on peut appliquer \ref{propr1lim} et on en déduit l’énoncé pour les faisceaux $\Of^{(r)}$ et $\Of^{**}$. On obtient aussi l'annulation de la cohomologie de $\G_m$ en degré supérieur ou égal à $2$.

D'après  \ref{claimostrelpolicd}, on a une décomposition en produits directs du système projectif\footnote{où $T$ est vu comme un système projectif constant} $(\Of^*(X_n))_n$  \[(\Of^*(X_n))_n=(L^*\Of^{**}(X_n))_n\times  (T)_n.\]  
Ainsi  $\Of^*({\rm Int}(\BC))=\limp_n L^*\Of^{**}(X_n)\times  \limp_n T=L^* \times  T$ (en utilisant \ref{propr1lim}) et \[{\rm Pic}_L({\rm Int}(\BC))=\rrr^1\limp_n L^*\Of^{**}(X_n)\times \rrr^1\limp_n T=0.\]
 



\end{proof}

Nous pouvons maintenant \'enoncer le th\'eor\`eme principal de cette section : 

\begin{theo}\label{theoacyco+arralg}
Les arrangements alg\'ebriques sont $\Of^{(r)}$, $\Of^{**}$-acycliques et les sections globales  sont constantes. 
\end{theo}

\begin{proof}

\begin{lem}\label{lemo+constzdt}
Les fibrations $Z^d_t$ sont acycliques pour $\Of^{(r)}$ et $\Of^{**}$, et les sections globales  sont constantes. 
\end{lem}

\begin{proof}
On raisonne sur la suite spectrale de Cech pour le recouvrement $f^{*}(\VC)$ de $Z^d_t$. D'apr\`es le corollaire \ref{coroo+arralg}, chaque intersection est $\Of^{(r)}$-acyclique et on se ram\`ene \`a calculer la cohomologie de Cech sur le recouvrement $f^*(\VC)$ qui est isomorphe à $\check{\CC}^\bullet(Z^d_t,f^*(\VC),\OC_L^{(r)})$. Mais le nerf du recouvrement est le simplexe standard $\Delta^t$ de dimension $t$, qui est contractile. Ceci montre l'annulation de la cohomologie en degr\'e sup\'erieur ou \'egal \`a 1. On obtient aussi ais\'ement que $\Of^{(r)}(Z^d_t)= \OC_L^{(r)}$. On raisonne de même pour $\Of^{**}$.
\end{proof}

D'après  \ref{lemo+constzdt} et l'identification $\P^d_{rig, L}=Z_d^d$, la flèche d'inclusion $Z_t^d \to \P^d_{rig, L}$ induit alors des isomorphismes 
\[\han{s} (\P^d_{rig, L}, \Of^{(r)}) \cong \han{s} (Z_t^d, \Of^{(r)})\]
pour tout $s$ positif. D'o\`u l'annulation de $\han{s} (\P^d_{rig, L}, Z_t^d, \Of^{(r)})$. Alors la suite spectrale \eqref{eqsuitspectgen} d\'eg\'en\`ere et on obtient $\han{s}(\P^d_{rig, L}, {\rm Int} (\AC), \Of^{(r)})=0$ pour tout $s$. Ce qui se traduit par 
\[ \han{s}({\rm Int} (\AC), \Of^{(r)}) = \begin{cases} \OC_L^{(r)} \text{ si } s=0 \\ 0 \text{ sinon}. \end{cases} \] 
On raisonne de même pour $\Of^{**}$.

\end{proof}

\section{Cohomologie analytique \`a coefficients dans $\Of^{(r)}$\label{sssectioncohoano+}}

\subsection{Cohomologie des fibrations $X^d_t(\beta)$.\label{sparagrapho+xdt}}

Nous allons chercher \`a d\'eterminer la cohomologie des espaces $X^d_t(\beta)$. Commen\c{c}ons par faire quelques rappels sur les faisceaux localement libres de rang 1 sur $\P^t_{rig,L}$ et $\P^t_{zar,L}$. Dans le cas alg\'ebrique, ils sont d\'ecrits par les faisceaux tordus $\Of_{\P^t_{zar,L}}(k)$ avec $k$ dans $\Z$. Ce faisceau se trivialise sur le recouvrement usuel $\VC$ et les fonctions de transition font commuter le diagramme\footnote{Dans ce qui suit, on note aussi $\tilde{z}$ la variable sur $\P^t_{zar,A}$ ie. $\P^t_{zar,A}={\rm Proj}(A[\tilde{z}_0,\cdots,\tilde{z}_t])$
} (cf \ref{fnotint} pour la notation $V_{\{i,j\}}$)
\begin{equation}\label{eq:trans}
\xymatrix{ 
\Of_{\P^t_{zar, L}} (k)|_{V_{\{i,j\}}} \ar[r]^-{\sim} \ar[d]^{\rm Id}  &  \Of_{V_i}|_{V_{\{i,j\}}} \ar[d]^{ m_{ (\frac{\tilde{z}_i}{\tilde{z}_j})^{-k}}}  \\
\Of_{\P^t_{zar, L}} (k)|_{V_{\{i,j\}}} \ar[r]^-{\sim} & \Of_{V_j}|_{V_{\{i,j\}}}.
}
\end{equation}
  
 En g\'eom\'etrie rigide, on peut encore d\'efinir les faisceaux tordus $\Of_{\P^t_{rig,L}}(k)$, $\Of^+_{\P^t_{rig,L}}(k)$ (version \`a puissance born\'ee) et $\Of^{(r)}_{\P^t_{rig,L}}(k)$ gr\^ace aux m\^emes morphismes de transition : 
  \[
\xymatrix{ 
\Of^+_{\P^t_{rig, L}} (k)|_{V(\beta)_{\{i,j\}}} \ar[r]^-{\sim} \ar[d]^{\rm Id}  &  \Of^+_{V(\beta)_i}|_{V(\beta)_{\{i,j\}}} \ar[d]^{ m_{ (\frac{\tilde{z}_i}{\tilde{z}_j})^{-k}}}  \\
\Of^+_{\P^t_{rig, L}} (k)|_{V(\beta)_{\{i,j\}}} \ar[r]^-{\sim} & \Of^+_{V(\beta)_j}|_{V(\beta)_{\{i,j\}}}.
}
\]  
On rappelle que l'on a bien $\frac{\tilde{z}_i}{\tilde{z}_j}\in \Of^+(V(\beta)_{\{i,j\}})$. On construit grâce à un diagramme similaire $\Of_{\P^t_{rig,L}}(k)$, $\Of^{(r)}_{\P^t_{rig,L}}(k)$.
 D'apr\`es GAGA (voir aussi \ref{theocohoangmxdt} pour une démonstration plus élémentaire), les $\Of_{\P^t_{rig,L}}(k)$ sont les seuls faisceaux localement libres de rang $1$. Nous pouvons aussi d\'efinir ces faisceaux tordus sur les fibrations $X^d_t(\beta)$ en tirant en arri\`ere par $f$ i.e. $\Of_{X^d_t(\beta)}(k) = f^{*} \Of_{\P^t_{rig,L}}(k)$, $\Of^+_{X^d_t(\beta)}(k) = f^{*} \Of^+_{\P^t_{rig,L}}(k)$ et $\Of^{(r)}_{X^d_t(\beta)}(k) = f^{*} \Of^{(r)}_{\P^t_{rig,L}}(k)$. L'un des points techniques de l'argument \ref{lemipiciso} montrera que $\Of_{X^d_t(\beta)}(k) = \Of_{\P^d_{rig,L}}(k)|_{X^d_t(\beta)}$.

En alg\'ebrique, la cohomologie de Zariski de ces faisceaux tordus est connue et peut \^etre trouv\'ee dans (\cite[ théorème 5.1 section III]{hart}) par exemple. Pour tout anneau $A$ et $k$ un entier, la cohomologie de $\Of_{\P^t_{zar,A}}(k)$ est concentr\'ee en degr\'e 0 si $k$ est positif et en degr\'e $t$ si $k$ est strictement n\'egatif. Plus pr\'ecis\'ement, on a des isomorphismes : 
\[ \hzar{0}(\P^t_{zar,A},\Of(k)) \cong A[ T_0, \dots ,T_t]_k \text{ si $k$ est positif}\]
\[ \hzar{t}(\P^t_{zar,A},\Of(k)) \cong (\frac{1}{T_0 \dots T_t} A[ \frac{1}{T_0}, \dots , \frac{1}{T_t}])_k \text{ si $k$ est n\'egatif}\]
o\`u $A[ T_0, \dots ,T_t]_k$ d\'esigne l'ensemble des polyn\^omes homogènes de degr\'e $k$. 

On se propose de calculer la cohomologie de $\Of^{(r)}$ des fibrations  $X^d_t(\beta)$. Plus précisément, nous souhaitons montrer :
\begin{theo}\label{coroorptxdt}

\begin{itemize}
\item La cohomologie à  coefficients dans $\Of^{(r)}(k)$ de l'espace projectif $\P^t_{rig,L}$ est concentrée en degré $0$ si $k$ est positif et en degré $t$ si $k$ est strictement négatif. De même la cohomologie des fibrations $X^d_t(\beta)$ est concentrée en degrés $0$ et $t$.
\item Plus pr\'ecis\'ement, on a des isomorphismes :   
\[\han{0}(\P^t_{rig,L},\Of^{(r)}(k)) \cong \OC^{(r)}_L\otimes_{\OC_L} \hzar{0}(\P^t_{zar,\OC_L},\Of(k))\cong \OC_L^{(r)} [ T_0, \dots ,T_t]_k, \]
\[ \han{t}(\P^t_{rig,L},\Of^{(r)}(k)) \cong \OC^{(r)}_L\otimes_{\OC_L} \hzar{0}(\P^t_{zar,\OC_L},\Of(k))\cong (\frac{1}{T_0 \dots T_t} \OC_L^{(r)}[ \frac{1}{T_0}, \dots , \frac{1}{T_t}])_k. \]
De plus, pour $r\le r'$, les flèches suivantes sont injectives pour $s\ge 0$
\[\han{s}(\P^t_{rig,L},\Of^{(r)}(k))\to \han{s}(\P^t_{rig,L},\Of^{(r')}(k)).\]
\item On dispose d'isomorphismes \[\han{0}(X^d_t(\beta),\Of^{(r)}(k))\simeq \drt{ \han{0}(\P^t_{rig,L},\Of^{(r)}(k-|\alpha|))}{\alpha\in \N^{d-t}\\ |\alpha|\le k}{},\]
 en particulier les sections globales des faisceaux $\Of^{(r)}(k)$ sont nulles si $k<0$ et s'identifie \`a $\OC_L^{(r)}$ si $k=0$. Enfin, 
 $\han{t}(X^d_t(\beta),\Of^{(r)}(k))$ est isomorphe au complét\'e $p$-adique de 
\[\drt{ \han{t}(\P^t_{rig,L},\Of^{(r)}(k-|\alpha|))}{\alpha\in \N^{d-t}\\|\alpha|\ge t+1+k}{}.\]
\end{itemize}


\end{theo}

\begin{proof}
Les intersections d'éléments du recouvrement $\VC(\beta)$ et $f^*(\VC(\beta))$ sont des produits de polycouronnes et de polydisques ferm\'es dont les polyrayons sont dans $|L^*|$. 
Ainsi, on se ramène à calculer la cohomologie de Cech sur les recouvrements $\VC(\beta)$ et $f^*(\VC(\beta))$ (cf \ref{theovdp}). De plus, pour toute section non nulle $h$ de $\Of^{(r)}(V(\beta)_I)$  (resp. $\Of^{(r)}(f^{-1}(V(\beta)_I))$) (cf \ref{fnotint} pour la notation), il existe une constante $\lambda\in \Of^{(r)}_L$ telle que $h/\lambda$ soit de norme $1$. On en déduit $\Of^{(r)}(V(\beta)_I)=\Of^{+}(V(\beta)_I)\otimes_{\OC_L}\Of^{(r)}_L$ (resp. $\Of^{(r)}(f^{-1}(V(\beta)_I))=\Of^{+}(f^{-1}(V(\beta)_I))\otimes_{\OC_L}\Of^{(r)}_L$) donc \[\check{\CC}^{\bullet} (\P^t_{rig,L};\Of^{(r)}(k),f^*(\VC(\beta)))=\check{\CC}^{\bullet} (\P^t_{rig,L};\Of^{+}(k),f^*(\VC(\beta)))\otimes_{\OC_L}\Of^{(r)}_L,\]
\[\check{\CC}^{\bullet} (X^d_t(\beta);\Of^{(r)}(k),\VC(\beta))=\check{\CC}^{\bullet} (X^d_t(\beta);\Of^{+}(k),\VC(\beta))\otimes_{\OC_L}\Of^{(r)}_L.\] 
Par platitude, on obtient les isomorphismes au niveau des groupes de cohomologie 
\[\hcech{*} (\P^t_{rig,L};\Of^{(r)}(k),f^*(\VC(\beta)))=\hcech{*} (\P^t_{rig,L};\Of^{+}(k),f^*(\VC(\beta)))\otimes_{\OC_L}\Of^{(r)}_L,\]
\[\hcech{*} (X^d_t(\beta);\Of^{(r)}(k),\VC(\beta))=\hcech{*} (X^d_t(\beta);\Of^{+}(k),\VC(\beta))\otimes_{\OC_L}\Of^{(r)}_L\]
et l'injectivité des inclusions quand $r$ varie. Le reste de cette section sera consacré au calcul de ces groupes de cohomologie de Cech sur $\Of^+$. Cela repose sur le lemme général suivant :

\begin{lem}\label{abstrait}
  Soit $\CC^{\bullet}$ un complexe constitué de $\Z_p$-modules plats tel que  
  $\hhh^j(\CC^{\bullet})$ est sans 
  $p$-torsion pour tout $j$.  On a alors un isomorphisme naturel $\hhh^j(\widehat{\CC^{\bullet}})\simeq 
  \widehat{\hhh^j(\CC^{\bullet})}$ où les complétions considérées sont réalisées suivant la topologie $p$-adique.
\end{lem}

\begin{proof} Soit $A^j={\rm Im}(d^{j-1}: \CC^{j-1}\to \CC^j)$
 et $B^j=\ker(d^{j}: \CC^j\to \CC^{j+1})$. On a  une suite exacte 
 $0\to A^j\to B^j\to \hhh^j(\CC^{\bullet})\to 0$ d'où l'exactitude de $0\to A^j/p^n\to B^j/p^n\to \hhh^j(\CC^{\bullet})/p^n\to 0$ car  $\hhh^j(\CC^{\bullet})$ 
 est sans $p$-torsion par hypoth\`ese. Par Mittag-Leffler, on obtient  encore une suite exacte $0\to \widehat{A}^j\to \widehat{B}^j\to   \widehat{\hhh^j(\CC^{\bullet})}\to 0$. Il suffit donc de montrer que $\widehat{A}^j={\rm Im}(\hat{d}^{j-1}: \widehat{\CC}^{j-1}\to \widehat{\CC}^j)$ et 
 $\widehat{B}^j=\ker(\hat{d}^j: \widehat{\CC}^j\to \widehat{\CC}^{j+1})$. Comme $\CC^{j+1}$ (et donc $A^{j+1}$) est sans $p$-torsion, on a l'exactitude de la suite $0\to B^j/p^n\to \CC^j/p^n\to A^{j+1}/p^n\to 0$   d'où celle de
 \[
 0\to B^{j}/p^n \to \CC^j/p^n  \to \CC^{j+1}/p^n
 \]
car on a montré que $A^{j+1}/p^n\to B^{j+1}/p^n \to \CC^{j+1}/p^n$ est injective.  En passant à la limite projective dans les deux suites précédentes, on obtient $\widehat{B}^j=\ker \hat{d}^{j}$ et $\widehat{A}^{j+1}\cong\widehat{\CC}^{j}/\widehat{B}^{j}=\widehat{\CC}^{j}/\ker{\hat{d}^j}\cong {\rm Im} \,\hat{d}^j$. 
 

\end{proof}


\begin{coro} \label{lemo+ptrig}
La cohomologie de Cech de $\Of^+_{\P^t_{rig,L}}(k)$ sur le recouvrement $V(\beta)$ est concentr\'ee en degr\'e 0 si $k$ est positif et en degr\'e $t$ si $k$ est strictement  n\'egatif. Plus pr\'ecis\'ement, on a des isomorphismes :   
\[\hcech{0}(\P^t_{rig,L},\Of^+(k),V(\beta)) \cong \OC_L [ T_0, \dots ,T_t]_k \text{ si $k$ est positif},\]
\[ \hcech{t}(\P^t_{rig,L},\Of^+(k),V(\beta)) \cong (\frac{1}{T_0 \dots T_t} \OC_L[ \frac{1}{T_0}, \dots , \frac{1}{T_t}])_k \text{ si $k$ est n\'egatif}.\]
\end{coro}

\begin{proof}
D'après  la description des fonctions analytiques de norme spectrale au plus $1$ sur un polydisque ou une polycouronne, le complexe $\check{\CC}^\bullet (\P^t_{rig, L} ; \Of_{\P^t_{rig, L}}^+(k) , \VC(\beta))$ est la complétion $p$-adique du complexe $\check{\CC}^\bullet (\P^t_{zar, \OC_L} ; \Of_{\P^t_{zar, \OC_L}}(k) , \VC)$. 
Le lemme \ref{abstrait}
 montre alors que les groupes de cohomologie de $\check{\CC}^\bullet (\P^t_{rig, L} ; \Of_{\P^t_{rig, L}}^+(k) , \VC(\beta))$ s'identifient aux compl\'et\'es $p$-adiques des groupes de cohomologie $\hzar{*}(\P^t_{zar, \OC_L}, \Of_{\P^t_{zar, \OC_L}}(k))$. Comme ces derniers sont de type fini sur $\OC_L$, la compl\'etion est en fait inutile, ce qui permet de conclure. 
\end{proof}


\begin{coro}\label{theocohoano+xdt}
 Soit $k$ un entier. La cohomologie de Cech de  $\Of^+(k)$  sur $X^d_t (\beta)$ pour le recouvrement $f^*(\VC(\beta))$ est concentrée en degrés $0$ et $t$. De plus on dispose d'isomorphismes \[\hcech{0}(X^d_t (\beta), \Of^+(k), f^*(\VC(\beta)))\simeq \drt{\hzar{0}(\P^t_{zar,\OC_L},\Of(k-|\alpha|))}{\alpha\in \N^{d-t}\\ |\alpha|\le k}{},\]
 en particulier les sections globales des faisceaux $\Of^{(r)}(k)$ sont nulles  si $k<0$ et sont constantes si $k=0$. Enfin, 
 $\hcech{t}(X^d_t (\beta), \Of^+(k),f^*(\VC(\beta)))$ est isomorphe au complét\'e $p$-adique de 
\[\drt{\hzar{t}(\P^t_{zar,\OC_L},\Of(k-|\alpha|))}{\alpha\in \N^{d-t}\\|\alpha|\ge t+1+k}{}.\]
\end{coro}

\begin{proof}

Pour tout $i\in I\subset \left\llbracket 0, t \right\rrbracket$ fixé, on a une trivialisation $f^{-1}(V(\beta)_I)\cong V(\beta)_I\times \B_L^{d-t}(-\beta_{i})$ (cf \ref{fnotint} pour la notation $V(\beta)_I$). 
D'après \ref{lemcourrel}, tout $\lambda_I\in \Of^+(k)(f^{-1}(V(\beta)_I))$ admet une écriture unique qui dépends du choix de l'élément $i$  
\begin{equation}\label{eq:decompok}
\lambda_I=\som{\lambda^{(i)}_{I,\alpha}(z) \left(\frac{w^{(i)}}{\varpi^{-\beta_{i}}} \right)^\alpha}{\alpha\in\N^{d-t}}{}
\end{equation} 
où $z=[z_0,\cdots , z_t]$ désigne la variable de $V(\beta)_I$ vu comme un ouvert de $\P^t_{rig, L}$, $w^{(i)}=(w^{(i)}_1,\cdots, w^{(i)}_{d-t})$ est la variable de $\B^{d-t}(-\beta_{i})$, les sections 
$\lambda^{(i)}_{I,\alpha}$ sont dans $ \Of^+(k)(V(\beta)_I)$ et tendent vers $0$ $p$-adiquement. On déduit de cette décomposition le fait que le complexe de Cech $\Of^{+}_{X^d_t(\beta)}(k)$ sur le recouvrement $f^*(\VC(\beta))$ est la complétion de la somme directe $\bigoplus_{\alpha\in\N^{d-t}}\CC^\bullet_\alpha$ où chaque complexe $\CC^\bullet_\alpha$ est défini par
\[\CC^s_\alpha:=\drt{\left(\frac{w^{(i)}}{\varpi^{-\beta_{i}}} \right)^\alpha \Of^+(k)(V(\beta)_I}{|I|=s+1}{}).\]

D'après le  lemme \ref{abstrait}, il suffit de voir que la cohomologie des complexes $\CC^\bullet_\alpha$ coïncide avec celle des faisceaux $\Of_{\P^t_{zar, \OC_L}}(k-|\alpha|)$. Expliquons comment exhiber un tel isomorphisme. La relation $\frac{w^{(i)}}{\varpi^{-\beta_{i}}}=\frac{\tilde{z}_j}{\tilde{z}_i}\frac{w^{(j)}}{\varpi^{-\beta_{j}}}$ induit l'identité $\lambda^{(i)}_{I,\alpha}=(\frac{\tilde{z}_i}{\tilde{z}_j})^{k-|\alpha|}\lambda^{(j)}_{I,\alpha}$ et on peut voir $\lambda_{I,\alpha}:= (\lambda^{(i)}_{I,\alpha})_{i\in I}$ comme un élément de $\Of^+(k-|\alpha|)(V(\beta)_I)$.  On en déduit alors un isomorphisme pour tout $\alpha$ : \[\CC^\bullet_\alpha\cong\check{\CC}^\bullet (\P^t_{rig, L} ; \Of^+_{\P^t_{rig, L}}(k-|\alpha|) , \VC(\beta)).\] 
Le résultat est alors une conséquence du corollaire \ref{lemo+ptrig}.



\end{proof}
\end{proof}

\begin{rem}\label{remhtanxdt}

En fixant une trivialisation $f^{-1}(V(\beta)_I)\cong V(\beta)_I\times \B_L^{d-t}(-\beta_{i})=\B_{V(\beta)_I}^{d-t}(-\beta_{i})$ pour $i\in I= \left\llbracket 0, t \right\rrbracket$, on peut voir le groupe $\hcech{t} (X^d_t(\beta);\Of^+, f^*(\VC))$ comme un facteur direct de $\Of^+(\B_{V(\beta)_I}^{d-t}(-\beta_{i})))/\Of^+(V(\beta)_I)$
\end{rem}


\subsection{Cohomologie des complémentaires de tubes d'hyperplans\label{sparagrapho+arrtubfer}}
Nous pouvons maintenant déterminer la cohomologie de $\Of^{(r)}$ d'un arrangement $\AC$ tubulaire fermé d'ordre $n$. Nous souhaitons établir :
\begin{theo}\label{theoacyco+arrtubfer}
Les arrangements tubulaires fermés ${\rm Int}(\AC)$ sont $\Of^{(r)}$-acycliques.
\end{theo}

Cela découle du principe général suivant :
\begin{lem}\label{lemmayervietitere}
Soit $X$ un $L$-espace analytique et $\UC=\{U_i :i\in I\}$ une famille d'ouverts de $X$. Soit $\hhh$ une théorie cohomologique vérifiant la suite exacte longue de Mayer-Vietoris tel que pour toute famille finie $J\subset I$, les unions $\uni{U_i}{i\in J}{}$ n'ont pas de cohomologie en degré supérieur ou égal à  $|J|$. Sous ces hypothèses, toutes les intersections finies non vides $\inter{U_i}{i\in J}{}$  sont acycliques pour la cohomologie $\hhh$.
\end{lem}
Il est à  noter que d'après \ref{coroorptxdt}, les complémentaires des voisinages tubulaires ouverts d'hyperplans vérifient les hypothèses pour $\hhh$  la cohomologie analytique à  coefficients dans $\Of^{(r)}$. En effet, pour $\AC$ un arrangement tubulaire fermé, la cohomologie d'un espace de la forme ${\rm Uni}(\BC)$, avec $\BC\subset\AC$, s'annule en degré supérieur ou égal à  ${\rm rg}(\BC)	\le |\BC|$.

\begin{rem}

Notons que pour appliquer le résultat \ref{lemmayervietitere}, nous avons seulement utilisé le fait que la cohomologie des fibrations $X^d_t(\beta)$ était concentrée entre les degrés $0$  et $t$. Cette propriété se déduit directement du théorème \ref{theovdp} par comparaison avec la cohomologie de Cech sur le recouvrement $f^*(\VC (\beta))$. Nous pouvons alors nous passer du calcul explicite de ces groupes qui constituent le cœur technique de la preuve de \ref{coroorptxdt}. Toutefois, la description qui en découle servira de manière cruciale dans la preuve du lemme \ref{lemobsarrtubferor}.

\end{rem}

\begin{proof} On peut supposer que $I=\left\llbracket 1,n\right\rrbracket $ et on raisonne par r\'ecurrence sur 
$n$, le cas $n=1$ \'etant \'evident. Supposons que le r\'esultat est vrai pour 
$n-1$. Il suffit de d\'emontrer l'acyclicit\'e de $Y=\bigcap_{i=1}^n U_i$ (les autres intersections 
 \'etant trait\'ees par l'hypoth\`ese de r\'ecurrence). Notons $V_i=U_i\cap U_n$ pour $1\leq i\leq n-1$ et observons que $Y=V_1\cap...\cap V_{n-1}$. Il suffit donc (gr\^ace \`a l'hypoth\`ese de r\'ecurrence) de montrer que la cohomologie de $\bigcup_{i\in J} V_i$ s'annule en degr\'e 
sup\'erieur ou \'egal \`a $|J|$ quand $J\subset \left\llbracket 1,n-1\right\rrbracket $. Soit donc 
$k\geq |J|$ et $V^J=\bigcup_{i\in J} V_i=U^J\cap U_n$, o\`u $U^J=\bigcup_{i\in J} U_i$. Une partie de la suite de Mayer-Vietoris 
s'\'ecrit 
   $$   \hhh^k(U^J\cup U_n)\to \hhh^k(U^J)\oplus \hhh^k(U_n)\to \hhh^k(V^J)\to \hhh^{k+1}(U^J\cup U_n).$$
   Puisque $k+1\geq |J\cup\{n\}|$, le terme $\hhh^{k+1}(U^J\cup U_n)$ s'annule par hypoth\`ese, et 
   il en est de m\^eme de $\hhh^k(U_n)$ et $\hhh^k(U^J)$, donc aussi de $\hhh^k(V^J)$, ce qui permet de conclure.

\end{proof}

Nous pouvons aussi tirer des informations importantes sur les sections globales à puissances bornées des arrangements ${\rm Int}(\AC)$. Nous commencerons par ce lemme général.

\begin{lem}\label{lemgam}

Soit $X$ un espace analytique, $\Ff$ un faisceau en groupes abéliens et   $\UC=\{U_i\}$  une famille d'ouverts de $X$ tel que toute intersection finie\footnote{cf \ref{fnotint} pour la notation $U_I$} $U_I$ est $\Ff$-acyclique. Dans ce cas, on a :\[\Ff (U_I)=\sum_{J\in E_I}r_{J,I}(\Ff (U_J))\]
où $E_I=\{J\subset I: J\neq \emptyset \et \han{|J|-1}(\bigcup_{j\in J}U_j,\Ff)\neq 0\}$ et $r_{J,I}:\Ff (U_J)\to \Ff (U_I)$ est la flèche de restriction.

\end{lem}

\begin{proof}

On raisonne par récurrence sur le cardinal de $I$. Le résultat est trivial quand ce dernier vaut 1. Fixons $I$ et supposons le résultat pour toute partie stricte de $I$. Si $\han{|I|-1}(\bigcup_{i\in I}U_i,\Ff)\neq 0$, c'est tautologique car $I\in E_I$. Sinon, on a par hypothèse
\[\han{|I|-1}(\bigcup_{i\in I}U_i,\Ff)=\hcech{|I|-1}(\bigcup_{i\in I}U_i,\{U_i:i\in I\},\Ff)=\Ff (U_I)/\sum_{i\in I}r_{I\backslash\{i\},I}(\Ff (U_{I\backslash\{i\}}))=0.\]
Mais par hypothèse de récurrence, \[\sum_{i\in I}r_{I\backslash\{i\},I}(\Ff (U_{I\backslash\{i\}}))= \sum_{i\in I} \sum_{J\in E_{I\backslash\{i\}}}r_{J,I}(\Ff (U_J))=\sum_{J\in E_I}r_{J,I}(\Ff (U_J))\] car $E_I=\bigcup_{i\in I} E_{I\backslash\{i\}}$. Le résultat s'en déduit.

\end{proof}

\begin{coro}[Décomposition en éléments simples]\label{corodecsimp}

Soit $\AC$ un arrangement tubulaire fermé, on a 
\[\Of^{(r)}({\rm Int}(\AC))=\som{\Of^{(r)}({\rm Int}(\BC))}{\BC\subset\AC\\ |\BC|={\rm rg}(\BC)<d+1}{}.\]

\end{coro}

\begin{proof}
On reprend les notations du lemme précédent. On remarque  l'identité\footnote{Notons que lorsque l'on a l'égalité $|\BC|={\rm rg}(\BC)=d+1$, on a ${\rm Uni(\BC)}=\P_{rig, L}^d$ qui est $\Of^{(r)}$-acyclique
d'après \ref{lemo+ptrig}} $E_{\AC}=\{\BC\subset\AC : |\BC|={\rm rg}(\BC)<d+1\}$ d'après \ref{theocohoano+xdt} et on conclut.
\end{proof}

\begin{lem}\label{lemobsarrtubferor}
Soit $\AC_{n}$ un arrangement tubulaire fermé d'ordre $n>d$ et $\AC_{n-d}$  la restriction de $\AC_n$ d'ordre $n-d$. On a l'inclusion : \[\Of^{(r)}({\rm Int}(\AC_{n}))\subset \OC_L^{(r)}+\varpi\Of^{(r)}({\rm Int}(\AC_{n-d})).\] 

\end{lem}

\begin{proof}
D'après le résultat précédent, on peut supposer $|\AC_{n}|={\rm rg}(\AC_{n})<d+1$. On raisonne par récurrence sur $t={\rm rg}(\AC_{n})-1$. Plus précisément, nous montrons que pour tout arrangement tubulaire fermé $\BC_m$ d'ordre $m$ quelconque vérifiant ${\rm rg}(\BC_m)\le t+1$, on a l'inclusion $\Of^{(r)}({\rm Int}(\BC_{m}))\subset \OC_L^{(r)}+\varpi\Of^{(r)}({\rm Int}(\BC_{m-(t+1)}))$. 

Quand  $t=0$, cela découle du cas de la boule qui a été traité dans \ref{lemobspolycourdisq}.  Supposons l'énoncé vrai pour $t-1$ et montrons le résultat pour l'arrangement $\AC_{n}$ de rang $t+1$. Notons $\AC_{n-1}$ la projection de $\AC_n$ d'ordre $n-1$. On a des recouvrements naturels ${\rm Uni}(\AC_{n})=\bigcup_{H\in \AC_{n}}\mathring{H}(|\varpi|^{n})^c$ et ${\rm Uni}(\AC_{n-1})=\bigcup_{H\in \AC_{n}}\mathring{H}(|\varpi|^{n-1})$ de cardinal\footnote{Nous nous autorisons des répétitions dans le deuxième recouvrement.} $t+1$ que l'on notera $\AC_{n}^c$ et $\AC_{n-1}^c$. 

Les intersections d'éléments de $\AC_{n}^c$ ou de $\AC_{n-1}^c$ sont $\Of^{(r)}$-acycliques d'après le théorème \ref{theoacyco+arrtubfer} et on peut calculer la cohomologie des espaces ${\rm Uni}(\AC_{n})$ et ${\rm Uni}(\AC_{n-1})$ via les complexes de Cech sur ces recouvrements. Ces derniers sont concentrés entre les degrés $0$ et $t$, on en déduit un isomorphisme\footnote{Ici, $\delta$ désigne la différentielle du complexe $\ccech{\bullet}({\rm Uni}(\AC_n), \Of^{(r)},\AC_{n}^c)$.} $\han{t} ({\rm Uni}(\AC_n), \Of^{(r)})\cong \ccech{t}({\rm Uni}(\AC_n), \Of^{(r)},\AC_{n}^c)/\delta (\ccech{t-1}({\rm Uni}(\AC_n), \Of^{(r)},\AC_{n}^c))$ (idem pour $\AC_{n-1}$). De plus, ces deux recouvrements sont compatibles avec l'inclusion ${\rm Uni}(\AC_{n-1})\subset {\rm Uni}(\AC_{n})$, d'où une flèche entre les complexes de Cech qui induit le morphisme fonctoriel $\han{t} ({\rm Uni}(\AC_{n}), \Of^{(r)})\fln{\varphi^{(r)}}{} \han{t} ({\rm Uni}(\AC_{n-1}), \Of^{(r)})$. En explicitant ces complexes,  on obtient un 
diagramme commutatif dont les lignes horizontales sont exactes :

\begin{equation}
 \label{diagobsor}
\xymatrix{ 
\sum_{a\in \AC_n} \Of^{(r)}({\rm Int}(\AC_n\backslash\{a\}))\ar[d]^{}   \ar[r]^{}  &\Of^{(r)}({\rm Int}(\AC_n)) \ar[d]^{}  \ar[r]^{} & \han{t} ({\rm Uni}(\AC_n), \Of^{(r)}) \ar[d]^{} \ar[r]^{}    & 0 \\
\sum_{a\in \AC_{n-1}} \Of^{(r)}({\rm Int}(\AC_{n-1}\backslash\{a\}))\ar[r]^{}   &\Of^{(r)}({\rm Int}(\AC_{n-1})) \ar[r]^{} &\han{t} ({\rm Uni}(\AC_{n-1}), \Of^{(r)})\ar[r]^{}& 0}.
\end{equation} 
On veut montrer
 \begin{equation}\label{eqinclphi+}
 {\rm Im}(\varphi^{(r)})\subset\varpi\han{t} ({\rm Uni}(\AC_{n-1}), \Of^{(r)}).
 \end{equation}
 
 Quand ${\rm rg}(\AC_{n-1})<{\rm rg}(\AC_{n})$, l'inclusion est triviale car $\han{t} ({\rm Uni}(\AC_{n-1}), \Of^{(r)})=0$. 
 
 Si ${\rm rg}(\AC_{n-1})={\rm rg}(\AC_{n})$, on a des isomorphismes compatibles \[{\rm Uni}(\AC_{n})\cong X^d_t(\beta)\et {\rm Uni}(\AC_{n})\cong X^d_t(\tilde{\beta})\] avec $\tilde{\beta}:=\beta -(1,\cdots,1)$. D'après \ref{remhtanxdt}, $\han{t} ({\rm Uni}(\AC_{n}), \Of^{(r)})$   est un facteur direct de \[\Of^{(r)}( \B^{d-t}_{V(\beta)_I}(-\beta_i))/\Of^{(r)}(V(\beta)_I)\] pour $I=\left\llbracket 0,t\right\rrbracket$, $i\in I$ fixé (idem pour ${\rm Int}(\AC_{n-1})$). De plus, la flèche $\varphi^{(r)}$ est    induite par la restriction  naturelle (notons l'égalité $V(\beta)_I=V(\tilde{\beta})_I$)\[\Of^{(r)}( \B^{d-t}_{V(\beta)_I}(-\beta_i))/\Of^{(r)}(V(\beta)_I)\to \Of^{(r)}( \B^{d-t}_{V(\beta)_I}(-(\beta_i-1)))/\Of^{(r)}(V(\beta)_I)\]  et l'image de  $\varphi^{(r)}$ est contenue dans $\varpi \Of^{(r)}( \B^{d-t}_{V(\beta)_I}(-(\beta_i-1)))/\Of^{(r)}(V(\beta)_I)$ d'après \ref{lemobspolycourdisqrel} ce qui entraîne \eqref{eqinclphi+}. 
 
 D'après \eqref{diagobsor} et   \eqref{eqinclphi+}, on obtient, pour toute fonction $f\in \Of^{(r)}({\rm Int}(\AC_n))$, une décomposition dans $\Of^{(r)}({\rm Int}(\AC_{n-1}))$
\[f=\sum_{a\in \AC_{n-1}}f_a +g\]
avec $f_a\in\Of^{(r)}({\rm Int}(\AC_{n-1}\backslash\{a\}))$ et $g\in \varpi\Of^{(r)}({\rm Int}(\AC_{n-1}))$. Comme ${\rm rg}(\AC_{n-1}\backslash\{a\})<{\rm rg}(\AC_{n})$, on a par hypothèse de récurrence \[f_a\in\OC_L^{(r)}+\varpi\Of^{(r)}({\rm Int}(\AC_{n-1-t}\backslash\{a\}))\subset\OC_L^{(r)}+\varpi\Of^{(r)}({\rm Int}(\AC_{n-(t+1)}))\] ce qui établit le résultat.
\end{proof}
\begin{coro}\label{lemobsarrtubferadd}
Soit $\AC_{n}$ un arrangement tubulaire fermé d'ordre $n>2d$ et $\AC_{n-2d}$  la projection de $\AC_n$ d'ordre $n-2d$. On a l'inclusion : \[\Of^{+}({\rm Int}(\AC_{n}))\subset \OC_L+\varpi\Of^{
+}({\rm Int}(\AC_{n-2d}))\] 

\end{coro}

\begin{proof}

Cela découle de \ref{proprlimo+o++} et de \ref{lemobsarrtubferor}.

\end{proof}

\subsection{Cohomologie analytique à  coefficients dans $1+\Of^{++}=\Of^{**}$\label{sssectioncohoanostst}}

Nous allons maintenant nous intéresser aux faisceaux $1+\Of^{++}$ et démontrer un théorème d'acyclicité semblable au théorème \ref{theoacyco+arrtubfer}. Le résultat suivant est le point clé de cette section. C'est une application du logarithme tronqué qui permet d'étudier $1+\Of^{++}$ par le biais de $\Of^{(r)}$. Ce résultat a été énoncé dans (\cite[ 3.26 remarque fin de page 195]{vdp}) par Van Der Put. Nous allons donner les détails de la preuve.

\begin{lem}\label{lemoroetun}

Soit $X$ un espace rigide quasi-compact
\begin{enumerate}
\item Si $X$ est $\Of^{(r)}$-acyclique pour tout $0< r \leq 1 $, alors $X$ est $1+ \Of^{++}$-acyclique
\item Si la cohomologie de $X$ à  coefficients dans $\Of^{(r)}$ est concentrée en degrés $0$ et $t$ pour tout $0<r \leq 1 $, et la flèche naturelle $\han{t}(X,\Of^{(r)})\to \han{t}(X,\Of^{(r')})$ est injective pour $r'\geq r$, alors la cohomologie de $X$ à  coefficients dans $1+\Of^{++}$ est concentrée en degrés $0$ et $t$.
\end{enumerate}

\end{lem}



\begin{proof}
  On suppose que la cohomologie de $X$ à  coefficients dans $\Of^{(r)}$ est concentrée en degrés $0$ et $t$ pour tout $r>0$, et la flèche naturelle $\han{t}(X,\Of^{(r)})\to \han{t}(X,\Of^{(r')})$ est injective pour $r'\geq r$.   Le premier point s'en déduit quand $t=0$.  
Soit $s\notin \{0, t\}$, on veut l'annulation de la cohomologie de $1+\Of^{++}$ en degré $s$. Remarquons que $1+ \Of^{++}= \varinjlim_{ r \to 1^-} 1+ \Of^{(r)}$ donc (quasi-compacit\'e) :
\[ \han{s}(X ; 1+ \Of^{++})= \varinjlim_{ r \to 1^-} \han{s}(X ; 1+ \Of^{(r)}). \]
On fixe $r <1$. On a la suite exacte :
\[ 0 \to 1+ \Of^{(r^2)} \to 1+ \Of^{(r)} \to \Of^{(r)}/ \Of^{(r^2)} \to 0 \]
o\`u la surjection est donn\'ee par $(1+x) \mapsto x$. Par hypoth\`ese, $\Of^{(r)}/ \Of^{(r^2)}$ a une cohomologie analytique concentrée en degrés $0$ et $t$ d'apr\`es la suite exacte\footnote{Il est à  noter que nous avons utilisé l'hypothèse d'injectivité de $\han{t}(X,\Of^{(r)})\to \han{t}(X,\Of^{(r')})$ pour montrer l'annulation de $\han{t-1}(X ; 1+ \Of^{(r)})$} : 
\[ 0 \to \Of^{(r^2)} \to \Of^{(r)} \to \Of^{(r)}/\Of^{(r^2)} \to 0.\]
On a donc une surjection : 
\begin{equation}
\label{eqisorr2}
\han{s}(X ; 1+ \Of^{(r^2)}) \to \han{s}(X ; 1+ \Of^{(r)}).
\end{equation}

  Il suffit de prouver que $\han{s}(X ; 1+ \Of^{(r)})=0$
    pour $r$ petit. Si  $r<|p|$ et $\Vert x\Vert < r$, alors pour tout $n$, on a\footnote{Justifions la première inégalité. Par hypothèse, on a supposé que $p^{1+\varepsilon}$ divisait $x$ pour $\varepsilon$ un rationnel assez petit et on rappelle l'identité $v_p(n!)=\frac{n-s_p(n)}{p-1}\le n-1$ avec $s_p(n)$ la somme des chiffres dans  l'écriture de $n$ en base $p$ d'où $\Vert\frac{x^{n-1}}{n!}\Vert <1$. On en déduit alors $\Vert \frac{x^n}{n!}\Vert <\Vert x\Vert$.} \[\Vert \frac{x^n}{n!}\Vert <\Vert x\Vert  \et \Vert \frac{x^n}{n}\Vert <\Vert x\Vert .\] Les s\'eries usuelles du logarithme et de l'exponentielle sont bien définies et vérifient \[\Vert \exp(x)-1\Vert =\Vert \log (1+x)\Vert =\Vert x\Vert \] et elles induisent des morphismes inverses l'un de l'autre
\[\exp : \Of^{(r)} \xrightarrow{\sim} 1+\Of^{(r)}   \et \log :1+ \Of^{(r)} \xrightarrow{\sim} \Of^{(r)},\]
d'o\`u l'annulation de $\han{s}(X ; 1+ \Of^{(r)})$ par hypoth\`ese.  

\end{proof}


Les estimés \ref{theoacyco+arrtubfer}, \ref{lemo+ptrig}, \ref{theocohoano+xdt} et \ref{coroorptxdt} de la section \ref{sssectioncohoano+} permettent d'appliquer directement ce résultat et d'obtenir (on utilise partout la topologie analytique): 
\begin{coro}\label{coroostst}
On a
\begin{enumerate}
\item Les arrangements tubulaires fermés ${\rm Int(\AC)}$ sont $\Of^{**}$-acycliques.

\item Les espaces projectifs $\P^d_{rig,L}$ sont $\Of^{**}$-acycliques.

\item La cohomologie des fibrations $X^d_t(\beta)$ pour le faisceau $\Of^{**}$ est concentrée en degrés $0$ et $t$.

\end{enumerate}
\end{coro}








Le résultat précédent nous permet d'appliquer \ref{lemgam} pour obtenir une version multiplicative de la décomposition en éléments simples \ref{corodecsimp}.

\begin{coro}[Décomposition en éléments simples]

Soit $\AC$ un arrangement tubulaire fermé, on a
\[\Of^{**}({\rm Int(\AC)})=\som{\Of^{**}({\rm Int(\BC)})}{\BC\subset\AC\\ |\BC|={\rm rg}(\BC)<d+1}{}.\]

\end{coro}

\section{Cohomologie analytique à  coefficients dans $\G_m$\label{sssectioncohoangm}}

\subsection{Cohomologie des fibrations $X_t^d(\beta)$\label{paragraphgmxdt}}

Nous souhaitons montrer le théorème suivant :
\begin{theo}\label{theocohoangmxdt}
Soit $s\geq 1$ et $t\ge 1$, les fonctions inversibles de $X_t^d (\beta)$ sont constantes et l'application $f^{*}$ en cohomologie donn\'ee par la fibration $f:X_t^d (\beta) \to \P^t_{ rig, L}$ induit une d\'ecomposition : 
\begin{equation}
\label{eqdecfibreGm}
\han{s} (X_t^d (\beta), \G_m)\cong \han{s} (X_t^d (\beta), \Of^{**})\times \han{s} (\P^t_{ rig, L}, \G_m).
\end{equation}
De plus, l'inclusion $\iota: X^d_t(\beta)\to \P^d_{ rig, L}$ induit une bijection entre $\han{*} (\P^d_{ rig, L}, \G_m)$ et le facteur direct $\han{*} (\P^t_{ rig, L}, \G_m)$. Enfin, pour tout corps $F$ on a une identification  \[\han{s} (\P^t_{rig, L}, \G_m)=\hzar{s} (\P^t_{zar, F}, \G_m)= \begin{cases}
\Z & {\rm si}\ s=1\\
\{0\} & {\rm si} \ s>1.
\end{cases}\]\end{theo}

\begin{coro}\label{corounigm}

Soit $\AC$ un arrangement tubulaire fermé, la cohomologie analytique à  coefficients dans $\G_m$ de ${\rm Uni}(\AC)$ s'annule en degré supérieur ou égal à ${\rm rg}(\AC)$

\end{coro}

\begin{proof}(\ref{theocohoangmxdt})
D'apr\`es \ref{theovdp} point 4., les recouvrements $\VC(\beta)$ et $f^*(\VC(\beta))$ ont des intersections $\G_m$-acycliques et on se ram\`ene \`a calculer la cohomologie de Cech sur ces recouvrements. Pour toute intersection $V(\beta)_{I}$, on fixe une trivialisation $f^{-1} (V(\beta)_I)\cong V(\beta)_I\times \B_L^{d-t}(-\beta_{i_0})$ pour $i_0\in I$ et on a d'après \ref{claimostrelpolicd} : 
\[\Of^* (V(\beta)_I)= L^* \Of^{**}(V(\beta)_I)\times T^{(t)}_I \] et 
\[\Of^* \left(f^{-1} (V(\beta)_I)\right)=L^*\Of^{**} \left(f^{-1} (V(\beta)_I)\right)\times T^{(t)}_I\]
 où $T^{(t)}_I=\left\langle \frac{\tilde{z}_i}{\tilde{z}_j}: \; i,j\in I\right\rangle_{\Z\modut}\subset \left\langle \frac{\tilde{z}_i}{\tilde{z}_j}: \; i,j\in \left\llbracket 0,t\right\rrbracket\right\rangle_{\Z\modut}$. 

Introduisons  le complexe $(\CC^i(T^{(t)}_\bullet))_{0\leq i\leq t}:=(\bigoplus_{I\subset \left\llbracket 0,t\right\rrbracket : |I|=i+1}  T^{(t)}_I)_i$  avec pour différentielles les sommes alternées des inclusions. C'est un facteur direct du complexe de Cech $\check{\CC}^\bullet (X^d_t(\beta);\G_m, f^*(\VC(\beta)))$ et on en déduit un isomorphisme \begin{equation}\label{eqcechct}
 \hcech{s} (X^d_t(\beta) ; \G_m, f^* (\VC(\beta)))\cong\hcech{s} (X^d_t(\beta) ; L^* \Of^{**},   f^*(\VC(\beta))\times \hhh^s (\CC^*(T^{(t)}_\bullet)).
\end{equation}
Nous allons, dans la suite de l'argument, calculer les deux termes apparaissant dans ce produit direct. Commençons par la cohomologie du complexe $\CC^*(T^{(t)}_\bullet)$.

\begin{prop}\label{propptzar}
On a pour tout corps $F$ 
\[\hzar{s} (\P^t_{zar,F} ; \G_m)\cong \hhh^s (\CC^*(T^{(t)}_\bullet))=\begin{cases}
\Z & {\rm si}\ s=1\\
\{0\} & {\rm si} \ s>1.
\end{cases}\]

\end{prop}

\begin{proof}
Nous allons procéder en calculant $\hzar{s} (\P^t_{zar,F} ; \G_m)$ de deux manières différentes. Dans un premier temps, nous utiliserons la suite exacte \eqref{eq:exdiv} (voir plus bas) pour donner une expression explicite à ces groupes de cohomologie. Puis nous étudierons la cohomologie de Cech de $\G_m$ sur le recouvrement standard $\VC=\{V_i\}$ de l'espace projectif pour relier $\hzar{s} (\P^t_{zar,F} ; \G_m)$ à la cohomologie du complexe $\CC^*(T^{(t)}_\bullet)$ (cet argument pourrait être vu comme un analogue algébrique de la décomposition \eqref{eqcechct}). 

Pour réaliser ces deux étapes, nous allons rappeler quelques propriétés du foncteur des diviseurs de Cartier. Nous renvoyons à \cite[II.6]{hart}, \cite[Chapitre 11 Paragraphe 9 à 14]{gorwed} pour leur étude sur un schéma localement factoriel. Nous aurons seulement besoin ici du cas où l'espace considéré $X=\spec (A)$ est affine de sections globales factorielles où la situation est beaucoup plus simple. Appelons $S$ l'ensemble des éléments irréductibles dans $A$ à  un inversible près et pour tout ouvert $U$ de $X$, $S(U)$ les éléments $f$ de $S$ tels que $V(f)$ ne rencontre pas $U$, intéressons-nous au faisceau (la propriété de recollement étant claire) 
\[{\rm Div} : U\subset X \mapsto \Z[S(U)].\] 

Les ouverts standards $D(f)=\spec (A[1/f])$  de $X$ sont encore affines de sections globales factorielles et l'ensemble des éléments irréductibles dans $A$ à  un inversible près est $S$ privé des facteurs irréductibles de $f$ à savoir $S(D(f))$. Si on note $\KC_X={\rm Frac}(A)$ les fractions rationnels sur $X$, on en  déduit l'identité 
\[{\rm Div}(D(f))=  \KC_X/\Of^*(D(f))\] par factorialité de $A[1/f]$. Comme les ouverts  standards forment une base de voisinage de $X$, on a la suite exacte de faisceaux
\begin{equation}\label{eq:exdiv}
0\to \G_m \to \KC_{X} \to {\rm Div} \to 0.
\end{equation}
De plus, ${\rm Div}$ est flasque par construction  et le faisceau constant $\KC_{X}$ l'est aussi  par irréductiblité de $X$.

Revenons à notre cas d'étude. Le recouvrement standard $\VC$ de l'espace projectif est constitué d'ouverts affines dont les sections globales sont factorielles. On peut construire comme dans la discussion précédente le faisceaux ${\rm Div}$ sur chacun de ces ouverts $V_i\in \VC$ et ces derniers se recollent sur l'espace projectif tout entier car chaque $V_i \cap V_j$ est encore affine de sections factorielles. Par recollement, on obtient encore une suite exacte
\[0\to \G_m \to \KC_{\P^t_{zar, F}} \fln{\pi}{} {\rm Div} \to 0\]
où les faisceaux $\KC_{\P^t_{zar, F}}$ et ${\rm Div}$ sont encore flasques\footnote{C'est une notion locale.
}; ils n'ont pas de cohomologie en degré strictement positif. Par suite exacte longue, on en déduit les égalités \[\hzar{s} (\P^t_{zar,F} ; \G_m)=0 \ {\rm si} \ s\geq  2\] \[\hzar{1} (\P^t_{zar,F} ; \G_m)={\rm Div}(\P^t_{zar, F})/\pi(\KC_{\P^t_{zar, F}})\cong\Z.\]
Nous laissons au lecteur le soin de vérifier que les sections globales de ${\rm Div}$ sur l'espace projectif s'identifie au  module libre sur $\Z$ engendré par les éléments irréductibles de $K[z_0, \cdots,z_d]$ à une unité près et que $\pi(\KC_{\P^t_{zar, F}})$ s'identifie au sous-ensemble des éléments de masse totale nulle. L'isomorphisme ci-dessus est alors induit par le degré.

On relie maintenant $\hzar{*} (\P^t_{zar,F} ; \G_m)$ à $\hhh^* (\CC^*(T^{(t)}_\bullet))$. On voit que chaque intersection $V_I:=\bigcap_{i\in I} V_i$ comme le spectre d'un anneau factoriel et on déduit de \eqref{eq:exdiv}
\[0\to \G_m \to \KC_{V_I} \to {\rm Div} \to 0\]
que chaque intersection $V_I$ est  $\G_m$-acyclique\footnote{En degré $1$, nous avons utilisé l'égalité ${\rm Div}(V_I)= \KC_X/\Of^*(V_I)$ qui découle de la construction de ${\rm Div}$.}. En particulier, on peut calculer $\hzar{*} (\P^t_{zar,F} ; \G_m)$ via la cohomologie de Cech sur le recouvrement $\VC$ qui admet la décomposition  \[\hcech{s} (\P^t_{zar,F} ; \G_m, \VC)=\hcech{s} (\P^t_{zar,F} ; F^*, \VC)\times\hhh^s (\CC^*(T^{(t)}_\bullet))=\hhh^s (\CC^*(T^{(t)}_\bullet)).\] La dernière égalité s'obtient  par contractibilité du nerf de $\VC$ ce qui termine l'argument


\end{proof}

Étudions maintenant le terme $\hcech{s} (X^d_t(\beta) ; L^* \Of^{**},   f^*(\VC(\beta))$ et décrivons un peu plus précisément le complexe de Cech associé. Pour toute partie $I$, on a une suite exacte de groupes :
\[ 1 \to 1+ \mG_L \to L^* \times \Of^{**} (f^{-1}(V(\beta)_I)) \to L^* \Of^{**} (f^{-1}(V(\beta)_I)) \to 1. \]
D'o\`u une suite exacte au niveau des complexes de Cech sur le recouvrement $f^*(\VC(\beta))$ et une suite exacte longue entre les cohomologies associées. Comme le nerf du recouvrement $f^*(\VC(\beta))$ est contractile, on a \footnote{\label{remcohofsccst}
Pour $s\ge 1$, cet argument prouve aussi l'annulation de $\hcech{s} (X^d_t(\beta) ; A, f^* (\VC(\beta)))$ pour tout faisceau constant $A$ ce qui établit $\han{s} (X^d_t(\beta) ; A)=0$ par \ref{theovdp}. Enfin, grâce à \ref{lemmayervietitere}, nous obtenons  pour tout arrangement tubulaire $\AC$ l'annulation $\han{s} ({\rm Int}(\AC), A)$.
} pour $s\geq 1$
\[\hcech{s} (X^d_t(\beta) ; 1+ \mG_L, f^* (\VC(\beta))) =  \hcech{s} (X^d_t(\beta) ; L^*, f^* (\VC(\beta)))=0.\]

On en déduit une suite d'isomorphisme : \[ \hcech{s} (X^d_t(\beta) ; L^* \Of^{**} , f^*(\VC(\beta))) \cong \hcech{s} (X^d_t(\beta) ; \Of^{**} , f^*(\VC(\beta))) \cong \han{s} (X^d_t(\beta) , \Of^{**}  ) \]
où le dernier isomorphisme s'obtient par acyclicit\'e des polycouronnes (\ref{theovdp} point 2. et \ref{lemoroetun}).

 Revenons à la cohomologie des fibrations. Dans le cas particulier où $t=d$ et $f= {\rm Id }$, on traite alors le cas de l'espace projectif. D'après \eqref{eqcechct}, \ref{coroostst} point 2. et \ref{propptzar}, on a des  isomorphismes  pour $s\geq 1$
\[   \han{s} (\P^d_{ rig, L} ; \G_m)\cong \hhh^s (\CC^*(T^{(d)}_\bullet))=\begin{cases}
L^* & {\rm si}\ s=0\\
\Z & {\rm si}\ s=1\\
\{0\} & {\rm si} \ s>1.
\end{cases} \]

\begin{rem}\label{remct}

L'identification ci-dessus entre les groupes de degré 1 est explicite en termes de fibrés  en droite. On voit que la famille de fonction inversible $(\frac{\tilde{z}_i^k}{\tilde{z}_j^k})_{0\le i,j\le d}$ apparaissant dans le diagramme \eqref{eq:trans} forme un cocycle dans $\CC^*(T^{(d)}_\bullet)$. Cela prouve que l'isomorphisme $\Z\cong \pic (\P^d_{ rig, L})$ est donné par $k\mapsto \Of_{\P^d_{ rig, L}}(k)$.

\end{rem}

Dans le cas général, on voit que la fibration $f$ identifie les facteurs directes isomorphes à $\CC^*(T^{(d)}_\bullet)$ dans les complexes $\check{\CC}^\bullet (X^d_t(\beta);\G_m, f^*(\VC(\beta)))$ et $\check{\CC}^\bullet (\P^t_{ rig, L};\G_m, \VC(\beta))$. La décomposition \eqref{eqcechct} devient
\[\han{s} (X_t^d (\beta), \G_m)\cong \han{s} (X_t^d (\beta), \Of^{**})\times \han{s} (\P^t_{ rig, L}, \G_m)\]
pour $s\ge 1$. 
En degré $0$, on a \[\Of^*(X^d_t(\beta))=\hcech{0} (X^d_t(\beta) ; L^* \Of^{**},  f^*(\VC(\beta)))=L^*\Of^{**} (X^d_t(\beta))=L^*\]
car $\hhh^0 (\CC^*(T^{(t)}_\bullet))=0$ (cf \ref{propptzar}) et $\Of^{**} (X^d_t(\beta))=1+\mG_L$ (voir \ref{theocohoano+xdt}).

Nous terminons l'argument en caractérisant l'image de  $\han{*} (\P^d_{ rig, L}, \G_m)$ dans  $\han{*} (X_t^d (\beta), \G_m)$ dans la décomposition \eqref{eqcechct}.



\begin{lem}\label{lemipiciso} 
L'inclusion $\iota: X^d_t(\beta)\to \P^d_{ rig, L}$ induit un isomorphisme de $\han{*} (\P^d_{ rig, L}, \G_m)$ sur le facteur direct $\han{*} (\P^t_{ rig, L}, \G_m)$ de $\han{*} (X_t^d (\beta), \G_m)$.
\end{lem}

\begin{proof}


Il suffit de montrer l’isomorphisme en degré $1$ car les groupes en degrés supérieurs sont nuls et les sections constantes sont identifiées en degré $0$. On rappelle que l'espace 
\[ X^d_t(\beta)=\{z=[z_0,\cdots,z_d]\in \P_{rig, L}^d, \exists i\leq  t, \forall j\le d, | z_i|\geq |\varpi|^{\beta_i}|z_j| \}\]
admet un recouvrement
\[ X^d_t(\beta)=\bigcup_{i\le t}\{z=[z_0,\cdots,z_d]\in \P_{rig, L}^d,\tilde{z}_i\neq 0 \et \forall j\le t, | \tilde{z}_i|\geq |\tilde{z}_j| \}=\bigcup_{i\le t}f^{*}(V(\beta))_i.\]
De plus, on a cette famille d'ouvert pour $0\le i\le d$
\[V_i:=\{z=[z_0,\cdots,z_d]\in \P_{rig, L}^d,\tilde{z}_i\neq 0  \}\]
 qui recouvre l'espace projectif tout entier et qui vérifie $f^{*}(V(\beta))_i\subset V_i$. La famille de fonctions inversibles
 \[(\frac{\tilde{z}_i}{\tilde{z}_j})_{0\le i,j\le d}\in \prod_{i,j} \Of^*(V_i\cap V_j) =\ccech{1}(\P^d_{rig, L},\G_m, \VC)\] 
définit un cocycle et corresponds donc à un fibré en droite sur $\P^d_{rig, L}$. D'après la remarque \ref{remct}, la classe de ce fibré engendre $\han{1} (\P^d_{ rig, L}, \G_m)$. Par compatibilité des recouvrements $f^{*}(\VC(\beta))$ et $\VC$, la restriction de cette classe à $X_t^d (\beta)$ est donnée par le cocycle 
\[(\frac{\tilde{z}_i}{\tilde{z}_j})_{0\le i,j\le t}\in \prod_{i,j} \Of^*(f^{*}(V(\beta))_i\cap f^{-1}(V(\beta))_j) =\ccech{1}(X^d_t(\beta),\G_m, f^{*}(\VC(\beta)))\] 
qui, toujours par la remarque la remarque \ref{remct}, engendre le facteur direct $\han{1} (\P^t_{ rig, L}, \G_m)$ de $\han{1} (X_t^d (\beta), \G_m)$. Ceci conclut l'argument.


\end{proof}

\begin{rem}\label{remcommiotastfst}
On a en fait montr\'e un résultat plus fort; on a un diagramme commutatif : 
\[ \xymatrix{
\han{s} (X_t^d(\beta), \G_m) & \ar[l]_-{\iota^*} \han{s} ( \P^d_{rig, L}, \G_m) \ar[ld]^{\varphi^*}_{\sim} \\
\han{s} ( \P^t_{rig, L}, \G_m) \ar[u]^{f^*} & 
} \]
o\`u $\varphi$ est le morphisme de $\P^t_{rig,L}$ dans $\P^d_{rig,L}$ donn\'e par $[z_0, \dots , z_t] \mapsto [z_0, \dots, z_t, 0 \dots, 0]$. Par contre, les morphismes au niveau des espaces ne commutent pas.  
\end{rem}

Ainsi tous les points ont été d\'emontr\'es. 
\end{proof}


\subsection{Cohomologie des arrangements tubulaires fermés\label{paragraphdegsuitespec}}

Nous sommes maintenant en mesure de déterminer la cohomologie à coefficients dans $\G_m$ pour les arrangements tubulaires fermés et ainsi donner l'un des résultats principaux de cette article.


\begin{theo}\label{theocohoangmarrtubfer}
Pour tout arrangement tubulaire fermé   $\AC$, les intersections ${\rm Int} (\AC)$ sont $\G_m$-acycliques.  
\end{theo}

\begin{proof}
On peut appliquer le résultat  \ref{lemmayervietitere} car la cohomologie des unions ${\rm Uni} (\AC)$ à coefficients dans $\G_m$ s'annule en degré supérieur ou égal à ${\rm rg}(\AC)$ d'après \ref{corounigm} où ${\rm rg}(\AC)\le |\AC|$.
\end{proof}

Nous avons aussi un résultat de structure pour les fonctions inversibles d'un arrangement tubulaire fermé. Nous aurons besoin de quelques notations.

\begin{defi}

Si $S$ est un ensemble fini et $A$ est un anneau, on note le sous-ensemble $A[S]^0\subset A[S]$ du module libre sur $A$ engendré par $S$ constitué des éléments de masse totale nulle\footnote{ie. les éléments $\sum_{s\in S}a_s\delta_s$ tels que $\sum_{s\in S}a_s=0$}.

\end{defi}

\begin{rem}\label{remidentarralggen}
Si $\AC$ est arrangement tubulaire fermé, nous faisons le choix d'un système de représentants des éléments de $\AC$ par des éléments de $\HC$ puis par des vecteurs unimodulaires que l'on voit comme des formes linéaires $(l_a)_{a\in \AC}$. Cela permet d'identifier $\Z[\AC]^0$ au sous groupe de $\Of^*({\rm Int}(\AC))$
\[\left\langle \frac{l_a (z)}{l_b(z)}: \; a,b\in \AC\right\rangle_{\Z\modut}.\]

\end{rem}

\begin{theo}\label{theoostarrtubfer}
Soit $\AC$ un arrangement tubulaire fermé, 
on a un isomorphisme \[\Of^*({\rm Int} (\AC))/L^*\Of^{**}({\rm Int} (\AC))\simeq \Z[\AC]^0.\]
\end{theo} 

\begin{rem}\label{remisoost}

Le théorème précédent montre plus précisément que la composée \[\Z[\AC]^0\subset\Of^*({\rm Int} (\AC))\flsur \Of^*({\rm Int} (\AC))/L^*\Of^{**}({\rm Int} (\AC))\]  ne dépend pas du choix du système de représentants et est un isomorphisme

\end{rem}

\begin{proof}


Comme dans la remarque \ref{remidentarralggen}, pour tout voisinage tubulaire $a\in\AC$, on fixe une forme linéaire $l_a$ représentée par un vecteur unimodulaire encore noté $a$ telle que $a=\mathring{\ker(l_a)}(|\varpi^{n }|)$. 

On introduit un faisceau $T$ grâce aux suites exactes suivantes
\[0\to \Of^{**}\to\G_m\to \G_m/\Of^{**}\to 0\]
\[0\to L^*/(1+\mG_L)\to \G_m/\Of^{**}\to T\to 0\]
 D'après \ref{theocohoangmxdt}, les flèches naturelles $\han{s}({\rm Uni}(\AC), \Of^{**})\to \han{s}({\rm Uni}(\AC), \G_m)$ sont injectives et les espaces ${\rm Uni}(\AC)$ sont acycliques pour les faisceaux constants (Note \eqref{remcohofsccst}). On en déduit des isomorphismes
\[\han{s}({\rm Uni}(\AC), \G_m/\Of^{**})=\frac{\han{s}({\rm Uni}(\AC), \G_m)}{\han{s}({\rm Uni}(\AC), \Of^{**})}=\begin{cases}
L^*/(1+\mG_L)& \si s=0\\
\Z &\si s=1 \et |\AC|\neq 1\\
0&\sinon
\end{cases},\]
\[\han{s}({\rm Uni}(\AC), T)=\frac{\han{s}({\rm Uni}(\AC), \G_m/\Of^{**})}{\han{s}({\rm Uni}(\AC), L^*/(1+\mG_L))}=\begin{cases}
\Z &\si s=1 \et |\AC|\neq 1\\
0&\sinon
\end{cases}\]
d'où l'acyclicité des espaces ${\rm Int}(\AC)$ pour les faisceaux $\G_m/\Of^{**}$ et $T$ par \ref{lemmayervietitere}. En particulier, la cohomologie à coefficients dans $T$ de ${\rm Uni}(\AC)$ peut se calculer grâce au complexe de Cech sur le recouvrement $\AC^c$ constitué des complémentaires des voisinages tubulaires $a \in \AC$. On sait aussi que $\Of^{**}$ et les faisceaux constants (cf. \eqref{remcohofsccst}) n'ont pas de cohomologie en degré supérieur ou égal à $1$ sur ${\rm Int}(\AC)$ d'où les égalités :
\[(\G_m/\Of^{**})({\rm Int}(\AC))=\Of^*({\rm Int}(\AC))/\Of^{**}({\rm Int}(\AC))\et T({\rm Int}(\AC))=\Of^*({\rm Int}(\AC))/L^*\Of^{**}({\rm Int}(\AC)).\]
Nous chercherons à décrire les sections globales de $T$ sur ${\rm Int}(\AC)$. Montrons par récurrence sur $|\AC|$, que la flèche décrite dans \ref{remisoost} est un isomorphisme 
\[T({\rm Int}(\AC))\cong \Z[\AC]^0.\]
Le reste de l'argument consiste à relier les sections de $T$ sur ${\rm Int}(\AC)$ aux groupes 
$\han{s}({\rm Uni}(\AC), T)$. 

Quand $|\AC|=1$, l'espace ${\rm Int}(\AC)$ est une boule et on a directement $\Of^*({\rm Int}(\AC))=L^*\Of^{**}({\rm Int}(\AC))$.

Pour $|\AC|=2$, on note $l_a$, $l_b$ les deux formes linéaires associées. La suite exacte de Mayer-Vietoris établit un isomorphisme $T({\rm Int}(\AC))\cong \han{1}({\rm Uni}(\AC),T)\cong \Z$.  De plus, d'après la discussion précédente  la flèche surjective $\G_m\to T$  induit un diagramme commutatif 
\[
\begin{tikzcd}
  \Of^*({\rm Int}(\AC))  \ar[r, equal] \ar[d,  twoheadrightarrow] &  \ccech{1}( {\rm Uni}(\AC),  \AC^c,  \G_m)  \ar[r,  twoheadrightarrow] \ar[d] & \han{1}({\rm Uni}(\AC),  \G_m)  \ar[d,  "\wr"  ]   \\ 
  \Of^*({\rm Int}(\AC))/ L^* \Of^{**}({\rm Int}(\AC)) \ar[r, equal]&   \ccech{1}( {\rm Uni}(\AC),  \AC^c,  T)   \ar[r, equal] &  \han{1}({\rm Uni}(\AC),  T).
\end{tikzcd}
\]
Il suffit prouver que $\frac{l_a}{l_b}$ engendre $\pic({\rm Uni}(\AC))=\han{1}({\rm Uni}(\AC),  \G_m) $.  En suivant \ref{excard2}, on peut trouver un changement de variables tel que 
\[
\begin{cases}
a=e_0, \; b=e_1  \mbox{ ou } \\ 
a=e_0,\; b=e_0+\varpi^k e_1 \mbox{ avec }  0<k<n.
\end{cases}
\]
Dans le premier cas,  on a ${\rm Uni}(\AC)\cong X_1^d(n,n)$ et les recouvrements  $\VC(n,n)$  et $\AC^c$ coïncident.  Mais d'après \ref{remct},  on a directement \footnote{On peut aussi utiliser le fait que ${\rm Int}(\AC)$ est ici une polycouronne et déduire le résultat de \ref{claimostrelpolicd}. } $\hcech{1}({\rm Uni}(\AC),  \VC(n,n) ,\G_m)=\left(\frac{l_a}{l_b}\right)^{\Z}$.  Dans le second cas,  on introduit l'ouvert  $U\subset {\rm Uni}(\AC)$ (noté $\mathring{H}_{e_1}(|\varpi^{n-k}|)^c$ dans \ref{excard2}) défini par
\[
U=\{z\in \P^{d}(C) : |z_i|\leq |\varpi^{n-k} z_1| \}.
\] 
Suivant si on échange $a$ et $b$,  on a deux isomorphismes de ${\rm Uni}(\AC) $ vers $ X^d_1(n,n-k)$ et donc  deux recouvrements $\VC(n,n-k)^{(1)}= \{ a^c,  U  \}  $ et $\VC(n,n-k)^{(2)}= \{b^c , U \}$.    Les cohomologies de Cech sur ces deux recouvrements sont bien comprises grâce à \ref{propptzar}, \ref{remct} et on a 
\[
\hcech{1}({\rm Uni}(\AC), \VC(n,n-k)^{(1)},  \G_m) = \left( \frac{l_a}{l_{e_1}} \right)^{\Z} \et \hcech{1}({\rm Uni}(\AC),  \VC(n,n-k)^{(2)},  \G_m)= \left( \frac{l_b}{l_{e_1}} \right)^{\Z}. 
\] 
Le triplet $(\frac{l_a}{l_b},  \frac{l_b}{l_{e_1}},  \frac{l_{e_1}}{l_a})$ définit un $1$-cocyle sur le recouvrement $\{a^c,  b^c, U  \}$  dont l'image  engendre $\hcech{1}({\rm Uni}(\AC), \VC(n,n-k)^{(1)},  \G_m)$ d'après l'équation qui précède.  Ainsi, sa projection $(\frac{l_a}{l_b},  \frac{l_b}{l_{e_1}},  \frac{l_{e_1}}{l_a})\mapsto\frac{l_a}{l_b}$ sur la cohomologie sur le recouvrement $\AC^c$ engendre encore  
\[
\hcech{1}({\rm Uni}(\AC),  \AC^c,  \G_m) =   \han{1}({\rm Uni}(\AC),  \G_m),
\]
ceci conclut l'argument.


Si $|\AC|=3$, on se donne encore $l_a$, $l_b$, $l_c$ des formes linéaires associées aux voisinages tubulaires. Étudions le complexe $\check{\CC}^\bullet ({\rm Uni}(\AC),\AC^c,T)=\check{\CC}^\bullet$ qui calcule les groupes $\han{*}({\rm Uni}(\AC), T)$. Quand $s\neq 1$, tous ces groupes ainsi que $\check{\CC}^3 $ sont nuls et on obtient l'exactitude de la suite :
\begin{equation}\label{eq:exccech}
0\to   \han{1}({\rm Uni}(\AC),T)\to \check{\CC}^1 /\delta(\check{\CC}^0)\to  \check{\CC}^2 \to  0.
\end{equation}

On a $\check{\CC}^2 =T({\rm Int}(\AC))$ et mais aussi les identités suivantes d'après le cas de cardinal $1$ et $2$ :
\[\check{\CC}^0 =T({\rm Int}(\{a\}))\times T({\rm Int}(\{b\}))\times T({\rm Int}(\{c\}))=0,\]
\[\check{\CC}^1 =T({\rm Int}(\{b,c\}))\times T({\rm Int}(\{c,a\}))\times T({\rm Int}(\{a,b\}))=(\frac{l_b}{l_{c}})^\Z\times(\frac{l_c}{l_{a}})^\Z\times(\frac{l_a}{l_{b}})^\Z.\]
En remplaçant  les termes précédents dans la suite exacte \eqref{eq:exccech}, on obtient

\[0\to \Z\fln{\alpha}{} (\frac{l_b}{l_{c}})^\Z\times(\frac{l_c}{l_{a}})^\Z\times(\frac{l_a}{l_{b}})^\Z\fln{\beta}{} T({\rm Int}(\AC))\to 0\]
avec $\beta$ le produit des trois termes. Il suffit maintenant de prouver que ${\rm im}\ \alpha=\ker \beta$ coïncide avec le sous-groupe $G$ engendré par le triplet $(\frac{l_b}{l_{c}}, \frac{l_c}{l_{a}}, \frac{l_a}{l_{b}})$. On a clairement l'inclusion $G\subset \ker \beta$. Le quotient $\ker \beta/G$ est de torsion\footnote{C'est un quotient de $\Z$ par un sous-groupe non trivial.} et s'injecte dans le groupe $(\frac{l_b}{l_{c}})^\Z\times(\frac{l_c}{l_{a}})^\Z\times(\frac{l_a}{l_{b}})^\Z/G$ qui est sans torsion. On en déduit l'annulation de $\ker \beta/G=0$ ainsi que les isomorphismes $T({\rm Int}(\AC))\cong (\frac{l_b}{l_{c}})^\Z\times(\frac{l_c}{l_{a}})^\Z\times(\frac{l_a}{l_{b}})^\Z/G\cong \Z[\AC]^0$.


Maintenant $|\AC|\ge 4$, et supposons le résultat pour tout arrangement tubulaire $\BC$ tel que $|\BC|<|\AC|$. On note encore $\check{\CC}^\bullet$ le complexe $\check{\CC}^\bullet ({\rm Uni}(\AC),\AC^c,T)$. On connaît l'annulation des groupes de cohomologie $\han{|\AC|-1}({\rm Uni}(\AC),T)=\han{|\AC|-2}({\rm Uni}(\AC),T)=0$ d'où une suite exacte
\begin{equation}\label{eq:stex}
 \check{\CC}^{|\AC|-3}\to \check{\CC}^{|\AC|-2}\to \check{\CC}^{|\AC|-1}\to 0.
 \end{equation}
Mais par hypothèse de récurrence, on a

\[\check{\CC}^{|\AC|-3}=\bigoplus_{c,d\in\AC}T({\rm Int}(\AC\backslash\{c,d
\})=\bigoplus_{c,d\in\AC}\Z[\AC\backslash\{c,d
\}]^0,\]

\[\check{\CC}^{|\AC|-2}=\bigoplus_{c\in\AC}T({\rm Int}(\AC\backslash\{c
\})=\bigoplus_{c\in\AC}\Z[\AC\backslash\{c
\}]^0.\]
En remplaçant ces deux termes et en observant que $\check{\CC}^{|\AC|-1}=T({\rm Int}(\AC))$, la suite exacte \eqref{eq:stex} devient :

\[\bigoplus_{c,d\in\AC}\Z[\AC\backslash\{c,d
\}]^0\fln{\varphi}{}\bigoplus_{c\in\AC}\Z[\AC\backslash\{c\}]^0\to T({\rm Int}(\AC))\to 0.\]
Il reste à établir l'isomorphisme $\Z[\AC]^0\cong {\rm Coker}(\varphi)$. Chaque fraction $\frac{l_a}{l_b}$ peut se voir dans $\Z[\AC\backslash\{c,d\}]^0$ ou $\Z[\AC\backslash\{c\}]^0$ pour  $a,b,c,d\in \AC$ distincts. Pour les distinguer, nous introduisons la notation \[(\frac{l_a}{l_b})^{(c,d)}\in \Z[\AC\backslash\{c,d\}]^0 \et (\frac{l_a}{l_b})^{(c)}\in \Z[\AC\backslash\{c\}]^0.\] Chacune des familles $((\frac{l_a}{l_b})^{(c,d)})_{a,b,c,d\in \AC}$, $((\frac{l_a}{l_b})^{(c)})_{a,b,c\in \AC}$ engendre $\bigoplus_{c,d\in\AC}\Z[\AC\backslash\{c,d
\}]^0$ et $\bigoplus_{c\in\AC}\Z[\AC\backslash\{c
\}]^0$  respectivement. Le groupe ${\rm Im}(\varphi)$ est engendré par les éléments $\varphi((\frac{l_a}{l_b})^{(c,d)})=(\frac{l_a}{l_b})^{(c)}(\frac{l_b}{l_a})^{(d)}$. Ainsi, 
la flèche
\begin{align*}
\bigoplus_{c\in\AC}\Z[\AC\backslash\{c\}]^0 &\rightarrow \Z[\AC]^0\\
(\frac{l_a}{l_b})^{(c)} & \mapsto  \frac{l_a}{l_b}
\end{align*}
induit l'isomorphisme ${\rm Coker}(\varphi)\cong \Z[\AC]^0$ voulu.

\end{proof}




\section{Etude des arrangements algébriques généralisés\label{sssectioncohoanarralggen}}


\begin{theo}\label{theocohoanarralggen}
Si $\AC$ est un arrangement algébrique généralisé, alors ${\rm Int} (\AC)$ est acyclique pour les faisceaux $\Of^{(r)}$, $\Of^{**}$ et $\G_m$ en topologie analytique. Les sections sur ${\rm Int} (\AC)$ de $\Of^+$, $\Of^{(r)}$, $\Of^{**}$ sont constantes et on a une suite exacte : \[0 \to L^* \to \Of^{*}({\rm Int}(\AC)) \to \Z \llbracket \AC \rrbracket^0 \to 0.\] 
\end{theo}

\begin{proof}
Considérons la famille $(\AC_n)_n$ d'arrangements tubulaires fermés compatible définie dans (\ref{remarr}). On obtient alors un recouvrement croissant de ${\rm Int} (\AC)=\uni{{\rm Int} (\AC_{n})}{}{}$  qui en fait un espace analytique quasi-Stein. Si $\Ff$ est l'un des faisceaux $\Of^{(r)}$, $\Of^{**}$, $\G_m$, on a la suite exacte \[0\to R^1\varprojlim_n \han{s-1}({\rm Int} (\AC_{n}),\Ff)\to \han{s} ({\rm Int} (\AC),\Ff)\to \varprojlim_n \han{s}({\rm Int} (\AC_{n}),\Ff)\to 0.\]
Par acyclicité des arrangements tubulaires d'hyperplans \ref{theoacyco+arrtubfer}, \ref{coroostst}, \ref{theoostarrtubfer},  on a \[\han{s}({\rm Int}(\AC), \Ff)= \begin{cases} \varprojlim_n \Ff({\rm Int}(\AC_{n})) & \text{ si } s=0 \\ R^1\varprojlim_n \Ff({\rm Int}(\AC_{n})) & \text{ si } s=1 \\ 0 & \text{ si } s \ge 2. \end{cases}\]
On peut appliquer \ref{propr1lim} grâce au point technique  \ref{lemobsarrtubferadd} pour obtenir l'acyclicité de ${\rm Int}(\AC)$ pour  $\Of^{(r)}$, $\Of^{**}$. On en déduit aussi la description des sections globales de  $\Of^+$, $\Of^{(r)}$, $\Of^{**}$ ce qui donne en particulier une autre démonstration du résultat \cite[lemme 3]{ber5}. 

Pour  $\G_m$, on a une suite exacte de systèmes projectifs \ref{theoostarrtubfer}: \[ 0 \to (L^*\Of^{**}({\rm Int} (\AC_{n}))_n \to (\Of^*({\rm Int} (\AC_{n})))_n \to (\Z[\AC_{n}]^0)_n \to 0. \] 
En appliquant le foncteur $\varprojlim_n$, on obtient une suite exacte longue : 
\[ 0 \to L^* \to \Of^{*}({\rm Int}(\AC)) \to \Z \llbracket \AC \rrbracket^0 \to R^1 \varprojlim_n L^* \Of^{**}({\rm Int}(\AC_{n})) \to$$ 
$$ R^1 \varprojlim_n \Of^{*}({\rm Int}(\AC_{n})) \to R^1 \varprojlim_n \Z[\AC_{n}]^0. \]
On a  $R^1 \varprojlim_n L^* \Of^{**}({\rm Int}(\AC_{n})) = R^1 \varprojlim_n \Z[\AC_{n}]^0=0$ d'apr\`es la surjectivité de $\Z[\AC_{n+1}]^0 \to \Z[ \AC_n]^0$ et \ref{propr1lim}. Donc \[ {\rm Pic}_{L}({\rm Int}(\AC))= R^1 \varprojlim_n \Of^{*}({\rm Int}(\AC_{n}))=0 \]
et la suite suivante est exacte  \[0 \to L^* \to \Of^{*}({\rm Int}(\AC)) \to \Z \llbracket \AC \rrbracket^0 \to 0.\]

\end{proof}




\section{Quelques commentaires sur la cohomologie \'etale et de de Rham des arrangements d'hyperplans }

\subsection{Cohomologie \'etale $l$-adique et de de Rham\label{sssectioncohodretarr}}

En appliquant la suite exacte de Kummer, on obtient d'après \ref{theocohoangmarrtubfer}, \ref{theoostarrtubfer}, \ref{theocohoanarralggen}

\begin{coro}\label{corocomposth1etdr}
Soit $m$ un entier. On a les diagrammes : 

\begin{enumerate}
\item si $\AC_n$ est un arrangement tubulaire ferm\'e et $m$ est premier \`a $p$.
\[\xymatrix{
\Of^*( {\rm Int}(\AC_n))/L^*\Of^{**}({\rm Int}(\AC_n)) \ar[r]^-{\kappa} & \het{1}({\rm Int}(\AC_n),\mu_m)/\kappa(L^*) \\
 \ar[u] \Z[ \AC_n]^0  \ar[r]^-{} &\Z/ m\Z[ \AC_n]^0 \ar[u]^{\rotatebox{90}{$\sim$}}. }
\]
\item si $\AC$ est un arrangement alg\'ebrique g\'en\'eralis\'e
\[\xymatrix{
\Of^*( {\rm Int}(\AC))/L^* \ar[r]^-{ \kappa } & \het{1}({\rm Int}(\AC), \mu_m)/\kappa(L^*) \\
 \ar[u] \Z\left\llbracket \AC\right\rrbracket^0  \ar[r]^-{ } & \Z/ m\Z\left\llbracket \AC\right\rrbracket^0 \ar[u]^{\rotatebox{90}{$\sim$}}. }
\]  
\end{enumerate}

\end{coro}

\begin{proof}
Dans les deux cas, le groupe de Picard de ${\rm Int}(\AC)$ est trivial d'où un isomorphisme par suite de Kummer :
\[ \het{1}({\rm Int}(\AC), \Z/m \Z)\cong  \Of^*( {\rm Int}(\AC))/(\Of^*( {\rm Int}(\AC)))^m.\]
Dans le second cas, la suite exacte de \ref{theocohoanarralggen} devient :
\[0\to L^*/(L^*)^m\to \Of^*( {\rm Int}(\AC))/(\Of^*( {\rm Int}(\AC)))^m\to  \Z/ m\Z\left\llbracket \AC\right\rrbracket^0\to 0.\]
car $\Z\left\llbracket \AC\right\rrbracket^0$ est sans $m$-torsion. L'argument se termine en identifiant $L^*/(L^*)^m$ et $\kappa(L^*)$.

On raisonne de manière similaire dans le premier cas en étudiant la suite exacte :
\[0\to L^*\Of^{**}({\rm Int}(\AC_n))/(L^*\Of^{**}({\rm Int}(\AC_n)))^m\to \Of^*( {\rm Int}(\AC_n))/(\Of^*( {\rm Int}(\AC_n)))^m\to  \Z/ m\Z[ \AC]^0\to 0.\] 
Mais, $\Of^{**}({\rm Int}(\AC_n))$ est $m$-divisible quand  $m$ est premier \`a $p$ car la série formelle $(X-1)^{1/m}$ converge sur $\Of^{++}({\rm Int}(\AC_n))$. On obtient la suite d'identification qui conclut la preuve 
\[L^*\Of^{**}({\rm Int}(\AC_n))/(L^*\Of^{**}({\rm Int}(\AC_n)))^m\cong L^*/(L^*)^m\cong \kappa(L^*).\]
\end{proof}

\begin{prop}\label{proproharrtubouvfernn-1}
Soit $n$ un entier, $\AC$ un arrangement tubulaire ouvert  d'hyperplans d'ordre $n$ et $\tilde{\AC}$ sa projection ferm\'ee d'ordre $n-1$. Alors l'inclusion ${\rm Int} (\tilde{\AC}) \to {\rm Int} (\AC)$ induit un isomorphisme au niveau des groupes de cohomologie de de Rham\footnote{Tous les groupes de cohomologie de de Rham sont calculés sur le site surconvergent. Notons que cette notion coïncide avec la cohomologie usuelle dans le cas tubulaire ouvert où les espaces sont partiellement propres.} (de m\^eme pour la cohomologie \'etale $l$-adique pour $L=C= \hat{\bar{K}}$). 
\end{prop}

\begin{proof}
\'Ecrivons $\hhh$ l'une des deux cohomologies consid\'er\'ees (avec $L= C$ pour la cohomologie \'etale $l$-adique). La suite spectrale \eqref{eqsuitspectgen} calculant $\hhh$ pour l'arrangement $\AC$ (resp. $\tilde{\AC}$) sera not\'ee $E_j^{-r,s}(\AC)$ et (resp.  $E_j^{-r,s}(\tilde{\AC})$). Nous allons les comparer pour \'etablir le r\'esultat. 

Consid\'erons $\BC$ une  partie de $\AC$ et $\tilde{\BC}$ sa projection dans $\tilde{\AC}$. On a ${\rm rg}(\BC)={\rm rg}(\tilde{\BC})=t+1$. Alors il existe $\beta \in \N^{t+1}$ tel que ${\rm Uni}(\BC)\cong Y^d_t(\beta)$ et  ${\rm Uni}(\tilde{\BC})\cong X^d_t(\beta)$. Les deux  cohomologies $\hhh$ v\'erifient l'axiome d'homotopie ie. pour tout espace analytique $X$, on a des isomorphismes induits par les projections naturelles (voir l'axiome d'homotopie \cite[§2 axiom I)]{scst}) pour la boule ouverte. Pour  la boule fermée, voir la formule Kuenneth \cite[Lemma 2.1.]{GK2} ou \cite[Proposition 3.3.]{GK3} en de Rham et \cite[Lemma 3.3.]{ber6} en $l$-adique) :
\[\hhh^*(X\times\B)\cong\hhh^*(X)\cong\hhh^*(X\times\mathring{\B}).\]
Ainsi,   les fibrations  induisent des isomorphismes entre la cohomologie de $\P^t_{rig,L}$ et celles de $Y^d_t(\beta)$, $X^d_t(\beta)$ compatibles par commutativit\'e du diagramme : 
\[ \xymatrix{
X^d_t(\beta) \ar[rd] \ar[rr] & & Y^d_t(\beta) \ar[ld] \\ & \P^t_{rig,L} & }; \] 
la flèche horizontale étant l'inclusion naturelle.

Par somme directe, on obtient un isomorphisme entre les suites spectrales, d'o\`u le r\'esultat.
\end{proof}

\subsection{Cohomologie \'etale $p$-adique des arrangements alg\'ebriques d'hyperplans\label{sssectioncohoetpadiquearralg}}
Ici, $L=C$ et on verra ${\rm Int}(\AC)$ comme un $C$-espace analytique par extension des scalaires pour $\AC$ un arrangement d'hyperplans $K$-rationnels.
\begin{prop}\label{propcohoetpadiquearralg}
Soit $\AC$ un arrangement alg\'ebrique $K$-rationnel, on a un isomorphisme canonique 
\[ \het{*}( {\rm Int}(\AC), \Q_p) \otimes C \cong \hdr{*} ( {\rm Int}(\AC)). \]
\end{prop}

\begin{rem}
Le r\'esultat r\'ecent \cite[Theorem 5.1. ]{GPW3} semble sugg\'erer que l'on a encore le r\'esultat pour les arrangements alg\'ebriques g\'en\'eralis\'es. 
\end{rem}

\begin{proof}
Appelons $E_j^{-r,s}(ét)$ et $E_j^{-r,s}(dR)$ les suites spectrales calculant respectivement la cohomologie \'etale $p$-adique et la cohomologie de de Rham. Nous allons exhiber un isomorphisme canonique $E_j^{-r,s}(ét) \otimes C \to E_j^{-r,s}(dR)$. Consid\'erons alors une union ${\rm Uni} (\BC)$ et \'ecrivons la $Z^d_t$. Nous allons montrer 
\[ \het{*}(Z^d_t, \Q_p) \otimes C \cong \hdr{*} ( Z^d_t).\]

Appelons $\Lambda$ le faisceau constant $\Z/ p^n \Z$. D'apr\`es un r\'esultat de Berkovich (\cite[Lemme 2.2]{ber4}), pour tout espace analytique $S$, tout entier $m$ et $\phi : \A^m_{rig, S} \to S$, on a $R^i \phi_* \Lambda_{\A^m_{rig, S}}=0$ pour $i\geq 1$. On a alors, par la suite spectrale de Leray, pour toute intersection $f^{-1}(V_I)$ de $f^*(\VC)$, $R \psi_* \Lambda_{f^{-1}(V_I)}=R \psi_* \Lambda_{V_I}$ o\`u $\psi : X \to \spg(C)$ pour tout $C$-espace analytique $X$. Par Cech, on obtient que $R \psi_{*} \Lambda_{Z^d_t} = R \psi_{*} \Lambda_{\P^t_{rig, C}}$ d'o\`u un isomorphisme 
\[ \het{i}(Z^d_t, \Q_p) \cong \het{i}(\P^t_{rig, C}, \Q_p). \]

De plus, d'apr\`es (\cite[ théorème 7.3.2.]{djvdp}), on a un isomorphisme canonique $\het{i}(\P^t_{rig, C}, \Q_p) \cong \het{i}(\P^t_{zar, C}, \Q_p)$. Par \'etude du cas alg\'ebrique, on en d\'eduit que $\het{*}(\P^t_{rig, C}, \Q_p)\otimes C$ est engendré en tant que $C$-algèbre graduée par l'image du faisceau tordu $\Of(1)$ par l'application de Kummer ${\rm Pic}(\P^t_{rig, C})\to \het{2}(\P^t_{rig, C}, \Q_p)$. On construit alors un isomorphisme en identifiant les classes logarithmiques. Ces morphismes commutent bien aux diff\'erentielles de la suite spectrale. On en d\'eduit le r\'esultat \`a la convergence.

\end{proof}

\newpage

\bibliographystyle{alpha}
\bibliography{Cohoan_v4}

\end{document}